\documentclass[reqno]{amsart}
\usepackage{tikz}
\usetikzlibrary{patterns} 
\usepackage{pgfplots}
\usepackage{amsfonts,amsmath,amsthm,amssymb,mathrsfs}
\usepackage{color}
\usepackage{amsmath,amssymb,amsfonts,bm}
\usepackage{graphicx}
\usepackage{newunicodechar}
\usepackage{xcolor}
\usepackage{geometry}
\numberwithin{equation}{section}
\usepackage[pdfstartview=FitH,colorlinks=true]{hyperref}
\hypersetup{urlcolor=red, linkcolor=blue,citecolor=red}
\allowdisplaybreaks

\theoremstyle{plain}

\newtheorem{thm}{Theorem}[section]

\newtheorem{prop}[thm]{Proposition}
\newtheorem{cor}[thm]{Corollary}
\newtheorem{lemma}[thm]{Lemma}
\newtheorem{remark}[thm]{Remark}

\begin{document}
\title[Landau equation]
{The Gevrey Gelfand-Shilov regularizing effect\\ 
of the Landau equation with soft potential}

\author[X.-D. Cao, C.-J. Xu and Y. Xu ]
{ Xiao-Dong Cao \& Chao-Jiang Xu \& Yan Xu}

\address{Xiao-Dong Cao, Chao-Jiang Xu
\newline\indent
School of Mathematics and Key Laboratory of Mathematical MIIT,
\newline\indent
Nanjing University of Aeronautics and Astronautics, Nanjing 210016, China
\newline\indent
Yan Xu
\newline\indent
Department of Mathematical Sciences, Tsinghua University, Beijing 100084, China
}
\email{caoxiaodong@nuaa.edu.cn; xuchaojiang@nuaa.edu.cn; xu-y@mail.tsinghua.edu.cn}

\date{\today}



\begin{abstract}
This paper studies the Cauchy problem for the spatially inhomogeneous Landau equation with soft potential in the perturbative framework around the Maxwellian distribution. Under a smallness assumption on the initial datum with exponential decay in the velocity variable, we establish the optimal Gevrey Gelfand-Shilov regularizing effect for the solution to the Cauchy problem.

\end{abstract}

\maketitle

\section{Introduction}
The Cauchy problem for the spatially inhomogeneous Landau equation is given by 
\begin{equation}\label{1-1}
\left\{
\begin{aligned}
 &\partial_t F+v\ \cdot\ \partial_x F=Q(F,\ F),\\
 &F|_{t=0}=F_0,
\end{aligned}
\right.
\end{equation}
where $F=F(t, x, v)\ge0$ denotes the density distribution function at time $t\ge0$, with position $x\in\mathbb{T}^3$ and velocity $v\in\mathbb R^3$. The Landau collision operator $Q$, which is bilinear with respect to the velocity variable, is defined by 
\begin{equation*}
    Q(G, F)(v)=\sum_{j, k=1}^3\partial_j\bigg(\int_{\mathbb R^3}a_{jk}(v-v_*)[G(v_*)\partial_kF(v)-\partial_kG(v_*)F(v)]dv_*\bigg),
\end{equation*}
where the non-negative symmetric matrix $\left(a_{jk}\right)$ is given by  
\begin{equation}\label{matrix A}
   a_{jk}(v)=(\delta_{jk}|v|^2-v_jv_k)|v|^\gamma,\quad \gamma\ge-3.
\end{equation}
The parameter $\gamma$ leads to the classification of hard potential if $\gamma>0$, Maxwellian molecules if $\gamma=0$, soft potential if $-3<\gamma<0$ and Coulombian potential if $\gamma=-3$.

The Landau equation is one of the fundamental kinetic models, derived as the grazing collision limit of the Boltzmann equation~\cite{V2}. Extensive research has been conducted on the spatially homogeneous case, in which the distribution function is independent of the spatial variable. In a pioneering work, Desvillettes and Villani~\cite{DV} established the smoothness of solutions to the spatially homogeneous Landau equation with hard potentials. The analytic smoothing effects were later obtained in~\cite{CLX1, L-1}, while the Gevrey regularity was studied in~\cite{CLX2, CLX3}. Moreover, the analytic Gelfand–Shilov smoothing effect was proved in~\cite{LX-5} under a perturbative framework near the normalized global Maxwellian. For Maxwellian molecules, the existence, uniqueness, and smoothness of solutions were investigated in~\cite{V1}, under the assumption that the initial data have finite mass and energy. The analytic and Gelfand-Shilov regularity properties were subsequently studied in~\cite{LX, M-1, MX}. In the case of soft potentials, existence and uniqueness results can be found in~\cite{FG, V2, W}. Regarding regularity,~\cite{LX2} showed that solutions to the linear Landau equation with soft potentials exhibit analytic smoothing. The Gelfand–Shilov regularizing effect for moderately soft potentials was further addressed in~\cite{LX3}.


In this paper, we consider the linearization of the Landau equation \eqref{1-1} around the Maxwellian distribution
    $\mu(v)=(2\pi)^{-\frac32}e^{-\frac{|v|^2}{2}}$,
and the fluctuation of the density distribution function
    $F=\mu+\sqrt\mu f$.
Since $Q(\mu, \mu)=0$, the Cauchy problem \eqref{1-1} is reduced to the form
\begin{equation}\label{1-2}
\left\{
\begin{aligned}
 &\partial_tf+v\ \cdot\ \partial_x f+\mathcal Lf=\Gamma(f, f),\\
 &f|_{t=0}=f_0,
\end{aligned}
\right.
\end{equation}
with the initial condition $F_0=\mu+\sqrt\mu f_0$, here the nonlinear Landau operator $\Gamma$ is defined by
$$\Gamma(f, f)=\mu^{-\frac12}Q(\sqrt\mu f, \sqrt\mu f)$$
and the linear Landau operator $\mathcal{L}$ is decomposed as
\begin{equation*}
\mathcal{L}=\mathcal{L}_1+\mathcal{L}_2,\ \ \ \mbox{with}\ \ \    \mathcal L_1f=-\Gamma(\sqrt\mu, f), \quad \mathcal L_2f=-\Gamma(f, \sqrt\mu).
\end{equation*}

In the perturbative framework, Guo~\cite{G-1} established the global-in-time existence and uniqueness of solutions to the spatially inhomogeneous Landau equation in Sobolev spaces. In~\cite{CLL}, Chen, Desvillettes, and He investigated the smoothing effects for classical solutions. Duan, Liu, Sakamoto, and Strain~\cite{DLSR} proved the existence of solutions with mild initial data. Furthermore, the smoothing properties of weak solutions with initial data bounded by a Gaussian in the velocity variable were studied in~\cite{H-1}.

Under the setting of the perturbation near global equilibrium, the analytic smoothing effect for the nonlinear Landau equation with Maxwellian molecules and small initial data in $H_x^r(L_v^2)$ (with $r > \frac{3}{2}$) was established in~\cite{M-2}. Additionally, the analytic smoothing effect in both spatial and velocity variables for hard potentials has been discussed in~\cite{C-L-X}, while the analytic Gelfand–Shilov regularizing effect has been addressed in~\cite{X}.

Now, we introduce the function space. Let $\Omega\subset\mathbb R^{3}$ be an open domain. For $s > 0$, the Gevrey class $G^{s}(\Omega)$ consists of all smooth functions $u$ such that there exists a constant $C > 0$ satisfying
$$
\|\partial^{\alpha}_{x}u\|_{L^{2}(\Omega)}\le C^{|\alpha|+1} (\alpha!)^{s}, \quad \forall\alpha\in\mathbb N^{3}. 
$$
For $\sigma, \nu > 0$ with $\sigma + \nu \ge 1$, the Gelfand--Shilov space $S^{\sigma}_{\nu}(\mathbb{R}^n)$ consists of all smooth functions $u$ for which there exists a constant $C > 0$ such that
\begin{align}\label{GS}
\|x^{\beta} \partial^{\alpha}_{x} u\|_{L^{2}(\mathbb{R}^n)} \le C^{|\alpha| + |\beta| + 1} (\alpha!)^{\sigma} (\beta!)^{\nu}, \quad \forall\, \alpha, \beta \in \mathbb{N}^n.
\end{align}
So that, the function of  the Gelfand--Shilov space $S^{\sigma}_{\nu}(\mathbb{R}^n)$ is belongs to Gevrey class $G^\sigma(\mathbb R^n)$ with an exponential decay, such as
$$
e^{c_0 \langle x\rangle ^{\frac 1\nu}} u\in L^2(\mathbb R^n).
$$

Before stating our main result, we introduce some notations. For simplicity, we denote $L^{2}_{x,v} = L^{2}(\mathbb{T}^{3}_{x} \times \mathbb{R}^{3}_{v})$ and $H^3_x L^2_v = H^3(\mathbb{T}^3_x; L^2(\mathbb{R}^3_v))$. For some $c_{0} > 0$ and $0< b \le 2$, we denote
\[
\omega_{t}(v) = e^{\frac{c_{0}}{1 + t} \langle v \rangle^{b}}, \quad t \ge 0,\ v \in \mathbb{R}^{3},
\]
where $\langle v \rangle = (1 + |v|^2)^{\frac{1}{2}}$. We also define the weighted Sobolev space 
$$H^3_x L^2_v(\omega_{t})=\{f: \ \|f\|_{H^3_x L^2_v(\omega_{t})}^{2}=\sum_{|\alpha|\le3}\|\omega_{t}\partial^{\alpha}_{x}f\|_{L^{2}_{x, v}}^{2}<\infty\}.$$
Our main result is restricted to the case $-3<\gamma<0$ and stated as follows,

\begin{thm}\label{main result}
Assume that the initial datum $\|f_{0}\|_{H^{3}_{x}L^{2}_{v}(\omega_{0})}$ small enough, then the Cauchy problem \eqref{1-2} admits a unique solution satisfying $\omega_t f(t) \in G^{\sigma}\left(\mathbb{T}^3_x; S^{\sigma}_{\sigma}\left(\mathbb{R}^3_{v}\right)\right)$ for $t > 0$  with $\sigma = \max\left\{1, \frac{b - \gamma}{2b} \right\}$.
Moreover, for any $T > 0$ and $\lambda > 2\sigma$,
there exist constants $C, \ \tilde{C} > 0$ such that for any $\alpha, \tilde{\alpha}, \beta \in \mathbb{N}^3$, the following estimate holds:
\begin{align}\label{Gelfand-Shilov}
\left\|v^\beta \partial^\alpha_v \partial^{\tilde{\alpha}}_{x} f(t) \right\|_{H^{3}_{x}L^{2}_{v}(\omega_t)} \le 
\left( \left( \frac{C}{t^{\lambda + 1}} \right)^{|\tilde{\alpha}| + 1} 
\left( \frac{\tilde{C}}{t^{\lambda}} \right)^{|\alpha| + |\beta| + 1} 
\alpha! \, \tilde{\alpha}! \, \beta! \right)^{\sigma}, 
\quad  0 < t \le T.
\end{align}
\end{thm}

\begin{remark}\label{remark1.2}
The existence and uniqueness of the Landau equation in Sobolev space had been addressed in~\cite{G-1} for all $\gamma\ge-3$. The results of~\cite{CLX} show the solution of the Landau equation belongs to $C^{\infty}(]0, \infty[; \cap_{s\ge0}H^{\infty, s}_{x, v}(\mathbb T^{3}_{x}\times\mathbb R^{3}_{v}))$ with 
the initial datum $\|f_{0}\|_{H^{3}_{x}L^{2}_{v}}\ll 1$. Under the assumptions of Theorem~\ref{main result}, the proof of the existence of the weak solution is similar to that of Proposition 4.1 in~\cite{C-X-X}. 
\end{remark}
\begin{remark}
In~\cite{HJL}, He, Ji and Li established Gevrey regularity with the index $\max\left\{\frac{2-\gamma}{4s}, 1\right\}$ for the Boltzmann equation without angular cutoff of index $0<s<1$ for soft potentials, with a certain exponential weight $e^{a_{0}\langle v\rangle^{2}}$ assumption on initial datum. Our work uses a more general initial condition and obtains the Gevrey Gelfand-Shilov smoothing effect 
\begin{align*}
    \omega_{t}\, f(t) \in G^{\sigma}(\mathbb{T}^3_x;  S^{\sigma}_{\sigma}(\mathbb{R}^3_{v})), \quad
      0< t , \quad  \sigma = \max\left\{1, \frac{b - \gamma}{2b} \right\}.
\end{align*}
This indicates that the solution is in exponential decay for velocity variables,
\[
e^{\frac{c_0}{1+T} \langle v \rangle^b + c_1 t^{\lambda} \langle v \rangle^{\frac1\sigma}} f(t) \in H^{3}_{x} L^{2}_{v}, \quad 0 < t , 
\]
then it improves the decreasing rate concerning the initial date if $b<\gamma+2$. On the other hand, if $\gamma\ge -b$, we get the analyticity of velocity and position variables with an exponential decay of velocity variables, 
$$
\omega_{t}\, f(t)\in G^1(\mathbb T^3; S^1_1(\mathbb R^3)),\quad 0<t.
$$

To get everything rigorous, and in particular to take care of the loss of weight appearing on the initial data, we need to interpolate with $L^{2}$ space and obtain the regularity for $-2<\gamma<0$.
\end{remark}

\section{Methodology and preliminary results}\label{s2}

Throughout the paper, the notation $A \lesssim B$ denotes that there exists a constant $C > 0$ such that $A \leq C B$. The symbol $[\cdot, \cdot]$ indicates the commutator between two operators.
In the following, we denote the weighted Lebesgue spaces
\begin{equation*}
    \|\langle\cdot\rangle^r f\|_{L^p(\mathbb R^3)}=\|f\|_{L^p_{r}(\mathbb R^3)},\quad 1\le p\le\infty, \ r\in\mathbb R.
\end{equation*}
For the matrix $(a_{jk})$ defined in \eqref{matrix A}, we denote $\bar a_{jk}=a_{jk}*\mu$ and the norm
\begin{align*}
    ||f||^2_\sigma&=\int\left(\bar a_{jk}\partial_jf\partial_kf+\frac14\bar a_{jk}v_jv_kf^2\right)dv, \quad |||f|||^2=\sum_{|\alpha|\le3}\int_{\mathbb T^{3}_{x}}\|\partial^{\alpha}_{x}f(x, \cdot)\|^2_\sigma dx.
\end{align*}
From Corollary 1 of~\cite{G-1}, for $\gamma\ge-3$, there exists a constant $C_1>0$ such that 
\begin{align}\label{upper bound}
     ||f||^2_{\sigma}&\ge C_{1}\left(\|\langle \cdot\rangle^{\frac{\gamma}{2}}\nabla_{v} f\|^2_{L^{2}(\mathbb R^{3}_{v})}+\|\langle \cdot\rangle^{\frac{\gamma+2}{2}}f\|^2_{L^{2}(\mathbb R^{3}_{v})}\right).   
\end{align}
We now define the creation and annihilation operators, as well as the gradient associated with the operator
$
\mathcal{H} = -\Delta_v + \frac{|v|^2}{4},
$
as follows:
\begin{equation*}
    A_{\pm, k} = \frac{1}{2}v_k \mp \partial_{v_k}, \ (1 \le k \le 3), \quad A_{\pm}^{\alpha} = A_{\pm, 1}^{\alpha_1} A_{\pm, 2}^{\alpha_2} A_{\pm, 3}^{\alpha_3}, \ (\alpha \in \mathbb{N}^3), \quad \nabla_{\mathcal{H}_{\pm}} = (A_{\pm, 1}, A_{\pm, 2}, A_{\pm, 3}).
\end{equation*}
The Proposition 2.3 of~\cite{LX3} shows that for $-3<\gamma<0$,
\begin{align}\label{upper bound1}
\begin{split}
    ||f||^2_{\sigma}&\ge C_{1}\left(\|\langle \cdot\rangle^{\frac{\gamma}{2}}\mathbf P_{v}\nabla_{\mathcal H_{\pm}} f\|^2_{L^{2}(\mathbb R^{3}_{v})}+\|\langle \cdot\rangle^{\frac{\gamma+2}{2}}(\mathbf I-\mathbf P_{v})\nabla_{\mathcal H_{\pm}} f\|^2_{L^{2}(\mathbb R^{3}_{v})}\right)\ge C_{1}\|\langle \cdot\rangle^{\frac{\gamma}{2}}\nabla_{\mathcal H_{\pm}} f\|^2_{L^{2}(\mathbb R^{3}_{v})},
\end{split}    
\end{align}
where $\mathbf P_{v}$ is the projection to the vector $v=(v_{1}, v_{2}, v_{3})$ defined via
$$(\mathbf P_{v} G)_{j}=\sum_{k=1}^{3}G_{k}v_{k}\frac{v_{j}}{|v|^{2}}, \quad G=(G_{1}, G_{2}, G_{3}).$$

First, we recall two results that have been established in the existing literature. In what follows, we adopt the convention of implicit summation over repeated indices.
\begin{lemma}~\cite{LX3}\label{representations}
For $f, g\in\mathcal S(\mathbb R^{3}_{v})$, we have
$$\mathcal L_{1}f=A_{+, j}\left((a_{jk}*\mu)A_{-, k}f\right),\ \mathcal L_{2}f=-A_{+, j}\left(\sqrt\mu(a_{jk}*(\sqrt\mu A_{-, k}f))\right),$$
$$\Gamma(f, g)=A_{+, j}\left((a_{jk}*(\sqrt\mu f))A_{+, k}g\right)-A_{+, j}\left((a_{jk}*(\sqrt\mu A_{+, k}f))g\right).$$
\end{lemma}

\begin{lemma}~\cite{XX}\label{lemma2.1}
    Let $-3<\gamma<0$, then for any $0<\epsilon_1<1$, there exists a constant $C_{\epsilon_1}>0$ such that for any suitable function $f$
    \begin{align*}
       (1-\epsilon_1)\|f\|^2_{\sigma}\le (\mathcal L_1f,  f)_{L^2}+C_{\epsilon_1}\|f\|^2_{2, \frac{\gamma}{2}}.
    \end{align*}
\end{lemma}

We observe that the same argument gives us the following inequality
    \begin{align}\label{L1}
       (1-\epsilon_1)|||f|||^2\le (\mathcal L_1f,  f)_{H^{3}_{x}L^2_{v}}+C_{\epsilon_1}\|\langle v\rangle^{\frac\gamma2}f\|^2_{H^{3}_{x}L^{2}_{v}}.
    \end{align}

\bigskip
\subsection*{Idea of proof for main Theorem \ref{main result}} 
As in the case of hard potentials, we employ a family of auxiliary vector fields $H_{\delta}$, which were first introduced in~\cite{CLX}:
\begin{equation*}\label{vecM}
   H_\delta = \frac{1}{\delta+1} t^{\delta+1} \partial_{x_1} - t^{\delta} A_{+, 1},
\end{equation*}
where $\delta > 2\max\left\{1, \frac{b - \gamma}{2b} \right\}$ and $-3 < \gamma < 0$. Specifically, we have
     $[H_\delta,\, \partial_t + v \cdot \nabla_x] = \delta t^{\delta - 1} A_{+, 1}$.
More generally, by induction on $k$, we can obtain that  
\begin{align}\label{kehigher}
\forall\, k \geq 1,\quad [H_\delta^{k},\, \partial_t + v \cdot \nabla_x] = \delta k t^{\delta - 1} A_{+, 1} H_\delta^{k - 1}.
\end{align}

Let $\lambda > 2 \max\left\{1, \frac{b - \gamma}{2b} \right\}$, and define
\begin{equation}\label{delta12}
\delta_1 = \lambda, \quad \delta_2 = \left(1 - \frac{b}{b - \gamma}\right)\lambda + \frac{2b}{b - \gamma} \max\left\{1, \frac{b - \gamma}{2b} \right\}.
\end{equation}
It follows that $\delta_1 > \delta_2 > 2 \max\left\{1, \frac{b - \gamma}{2b} \right\}$. With these parameters, we define
\[
H_{\delta_1} = \frac{1}{\delta_1 + 1} t^{\delta_1 + 1} \partial_{x_1} - t^{\delta_1} A_{+, 1}, \quad
H_{\delta_2} = \frac{1}{\delta_2 + 1} t^{\delta_2 + 1} \partial_{x_1} - t^{\delta_2} A_{+, 1}.
\]
Then $[H_{\delta_1}, H_{\delta_2}] = 0$, and both $\partial_{x_1}$ and $A_{+, 1}$ can be expressed as linear combinations of $H_{\delta_1}$ and $H_{\delta_2}$:
\begin{equation}\label{linear combination-H}
\left\{
\begin{aligned}
t^{\lambda+1} \partial_{x_1} &= \frac{(\delta_2 + 1)(\delta_1 + 1)}{\delta_2 - \delta_1} H_{\delta_1} 
- \frac{(\delta_2 + 1)(\delta_1 + 1)}{\delta_2 - \delta_1} t^{\delta_1 - \delta_2} H_{\delta_2} 
:= \mathcal{T}_1 + \mathcal{T}_2, \\
t^{\lambda} A_{+, 1} &= \frac{\delta_1 + 1}{\delta_2 - \delta_1} H_{\delta_1}
- \frac{\delta_2 + 1}{\delta_2 - \delta_1} t^{\delta_1 - \delta_2} H_{\delta_2} 
:= \mathcal{T}_3 + \mathcal{T}_4.
\end{aligned}
\right.
\end{equation}
This decomposition allows us to control the classical directional derivatives along $H_{\delta_1}$ and $H_{\delta_2}$.

For $m+n\ge1$, by using \eqref{kehigher}, we have
\begin{align}\label{commutator-mn}
\begin{split}
     &\left([H^{m}_{\delta_{1}}H^{n}_{\delta_{2}}, \partial_t+v\ \cdot\ \partial_x], \ H^{m}_{\delta_{1}}H^{m}_{\delta_{2}}f\right)_{H^{3}_{x}L^{2}_{v}}\\
     &=\delta_{1}mt^{\delta_{1}-1}\left(A_{+, 1}H^{m-1}_{\delta_{1}}H^{n}_{\delta_{2}}, \ H^{m}_{\delta_{1}}H^{m}_{\delta_{2}}f\right)_{H^{3}_{x}L^{2}_{v}}+\delta_{2}nt^{\delta_{2}-1}\left(A_{+, 1}H^{m}_{\delta_{1}}H^{n-1}_{\delta_{2}}, \ H^{m}_{\delta_{1}}H^{m}_{\delta_{2}}f\right)_{H^{3}_{x}L^{2}_{v}}.
\end{split}
\end{align}
Since $\gamma<0$, the index in \eqref{upper bound1} and interpolation 
$$\|g\|^{2}_{L^{2}}\le\|\langle\cdot\rangle^{\frac\gamma2}g\|_{L^{2}}^{\gamma+2}\|\langle\cdot\rangle^{\frac\gamma2+1}g\|_{L^{2}}^{-\gamma},$$
implies that for the stronger case $-3\le\gamma\le-2$ cannot be estimated without a weighted function. So we introduce a weight $\omega_{t}(v) = e^{\frac{c_{0}}{1 + t} \langle v \rangle^{b}}$ with $t \ge 0,\ v \in \mathbb{R}^{3}$ and  $0<b\le2$. Then
\begin{align}\label{kehigher-1}
\begin{split}
\forall\ k\geq 1,\quad 	[\omega_{t} H_\delta ^k, \,\,  \partial_t+v\,\cdot\,\partial_x ]&=\omega_{t}[H_\delta^{k}, \,\, \partial_t+v\,\cdot\,\partial_x ]-\partial_t\omega_{t} H^{k}_\delta\\
&=\delta kt^{\delta-1} \omega_{t} A_{+, 1} H_\delta^{k-1}-\partial_t\omega_{t} H^{k}_\delta,
\end{split}
\end{align}

From \eqref{linear combination-H}, since $[\mathcal T_{j}, \mathcal T_{k}]=0$ for any $j, k$, we have that for all $\alpha_{1}, m\in\mathbb N$
\begin{align*}
	&t^{(\lambda+1)\alpha_{1}+ \lambda m} \|\omega_{t} {\partial_{x_{1}}^{\alpha_{1}} A_{+, 1}^{m}  f(t)}\|_{H^3_{x}L^2_v} =\|\omega_{t} (\mathcal T_{1}+\mathcal T_{2})^{\alpha_{1}}(\mathcal T_{3}+\mathcal T_{4})^{m}  f(t)\|_{H^3_{x}L^2_v}\\
	&\le\left|\frac{(\delta_2+ 1)(\delta_1+1)}{\delta_2-\delta_1}\right|^{\alpha_{1}+m}\sum_{j=0}^{\alpha_{1}}\sum_{k=0}^{m}\binom{\alpha_{1}}{j}\binom{m}{k}t^{(\delta_{1}-\delta_{2})(\alpha_{1}+m-j-k)}\|\omega_{t} H^{j+k}_{\delta_{1}}H^{\alpha_{1}+m-j-k}_{\delta_{2}}f(t)\|_{H^{3}_{x}L^{2}_{v}}.
\end{align*}
The above inequality together with Proposition 5.2 of~\cite{LX-5} and Theorem 2.1 of~\cite{GPR} can be used to obtain \eqref{Gelfand-Shilov}. So that to finish the proof of Theorem \ref{main result}, it suffice to show that there exists a constant $A > 0$ such that for any $0 < t \le T$ and any $m, n \in\mathbb N$,
\begin{align*}
     \|\omega_{t} H^{m}_{\delta_{1}}H^{n}_{\delta_{2}}f(t)\|_{H^{3}_{x}L^{2}_{v}}\le A^{m+n-\frac12}\left((m-2)!(n-2)!\right)^{\sigma}.
\end{align*}

Next, we review the commutator between the nonlinear Landau operator and weight $\omega_{t}$, which has been addressed in~\cite{C-X-X}.
\begin{lemma}~\cite{C-X-X}\label{commutator-nonlinear}
     Let $-3<\gamma<0$, then there exists a constant $C_{3}>0$,  which depends on $\gamma, \ b$ and $c_{0}$, such that for any suitable functions $f, g$ and $h$, 
\begin{align*}     
     \left|\left(\omega_{t}\Gamma(f, g), \omega_{t} h\right)_{L^{2}_{v}}\right|\le C_{3}\left\|f\right\|_{2, \frac\gamma2}\|\omega_{t} g\|_{\sigma}\|\omega_{t} h\|_{\sigma}.
\end{align*}     
\end{lemma}

Since $H^{3}_{x}$ is an algebra, which can be proved by using the Fourier transformation of $x$ variable, then we can extend the trilinear estimate into $H^{3}_{x}L^{2}_{v}$.
\begin{lemma}\label{commutator-nonlinear-x}
     Let $-3<\gamma<0$, then there exists a constant $C_{4}>0$,  which depends on $\gamma, \ b$ and $c_{0}$, such that for any suitable functions $f, g$ and $h$, 
\begin{align*}     
     \left|\left(\omega_{t}\Gamma(f, g), \omega_{t} h\right)_{H^{3}_{x}L^{2}_{v}}\right|\le C_{4}\left\|f\right\|_{H^{3}_{x}L^{2}_{v}}|||\omega_{t} g|||\cdot|||\omega_{t} h|||.
\end{align*}     
\end{lemma}

\section{Commutators between weights and Landau operators with vector fields}\label{s3}

This section is devoted to constructing some commutator estimates of the Landau operator, which will be used to prove our main result. 
We first review the following Leibniz-type formula.
\begin{lemma}~\cite{X}\label{Leibniz-type formula}
For all suitable functions $F$ and $G$ we have
$$H^{m}_{\delta}\left(a_{jk}*(\sqrt\mu F)G\right)=\sum_{l=0}^{m}\binom{m}{l}\left(a_{jk}*(\sqrt\mu H_{\delta}^{l}F)H_{\delta}^{m-l}G\right), \quad \forall \ m\ge1.$$
\end{lemma}

From Lemma \ref{commutator-nonlinear-x} and above the Leibniz-type formula, we can immediately obtain the following estimate of the nonlinear Landau operator.
\begin{prop}\label{H-m-Gamma}
For any $m, n\in\mathbb N$, let $-3<\gamma<0$, then for all $\delta_{1}, \delta_{2}>2\max\{1, \frac{b-\gamma}{2b}\}$ and any suitable functions $f, g, h$, we have
\begin{align*}
     &\left|\left(\omega_{t} H^{m}_{\delta_{1}}H^{n}_{\delta_{2}}\Gamma(f, g),\omega_{t} H^{m}_{\delta_{1}}H^{n}_{\delta_{2}}h\right)_{H^{3}_{x}L^{2}_{v}}\right|\\
     &\le C_{4}\sum_{l=0}^{m}\sum_{p=0}^{n}\binom{m}{l}\binom{n}{p}\|H^{l}_{\delta_{1}}H^{p}_{\delta_{2}}f\|_{H^{3}_{x}L^{2}_{v}}|||\omega_{t} H^{m-l}_{\delta_{1}}H^{n-p}_{\delta_{2}}g|||\cdot|||\omega_{t} H^{m}_{\delta_{1}}H^{n}_{\delta_{2}}h|||.
\end{align*}
\end{prop}

Now, we point out an estimate of the linear Landau operator $\mathcal L_{2}$; we begin with a singular integral in~\cite{L-1}. For any $s>-3$ and $\delta>0$, we have 
     \begin{equation}\label{singular integration}
          \int_{\mathbb R^{3}}|v-w|^{s}e^{-\delta|w|^{2}}dw\le C_{\delta, s}\langle v\rangle^{s}.
     \end{equation}

\begin{cor}\label{L2}
     For any $m, n\in\mathbb N_{+}$, we have for all $\delta_{1}, \delta_{2}>2\max\{1, \frac{b-\gamma}{2b}\}$ and any suitable function $f$
\begin{align*}     
     &\left|\left(\omega_{t} H^{m}_{\delta_{1}}H^{n}_{\delta_{2}}\mathcal L_{2}f, \omega_{t} H^{m}_{\delta_{1}}H^{n}_{\delta_{2}}f\right)_{H^{3}_{x}L^{2}_{v}}\right|\le C_{5}\sum_{l=0}^{m}\sum_{p=0}^{n}\binom{m}{l}\binom{n}{p}\left(t^{\delta_{1}}\sqrt{C_{0}}\right)^{m-l}\left(t^{\delta_{2}}\sqrt{C_{0}}\right)^{n-p}\\
     &\qquad\qquad\qquad\times\sqrt{(m-l+n-p+3)!}\|H^{l}_{\delta_{1}}H^{p}_{\delta_{2}}f\|_{H^{3}_{x}L^{2}_{v}}|||\omega_{t} H^{m}_{\delta_{1}}H^{n}_{\delta_{2}}f|||.
\end{align*} 
with the constants $C_{0}, C_{5}>0$ are independent of $m$ and $n$, but depends on $\gamma, \ b$ and $c_{0}$.
\end{cor}
\begin{proof}
     Since $\mathcal L_{2}f=-\Gamma(f, \sqrt\mu)$, we have
\begin{align*}     
     &\left(\omega_{t} H^{m}_{\delta_{1}}H^{n}_{\delta_{2}}\mathcal L_{2}f, \omega_{t} H^{m}_{\delta_{1}}H^{n}_{\delta_{2}}f\right)_{H^{3}_{x}L^{2}_{v}}\\
     &=\sum_{|\alpha|\le3}\sum_{l=0}^{m}\sum_{p=0}^{n}\binom{m}{l}\binom{n}{p}\int_{\mathbb T^{3}_{x}}\left(\omega_{t}\Gamma(H^{l}_{\delta_{1}}H^{p}_{\delta_{2}}\partial^{\alpha}_{x}f, \ H^{m-l}_{\delta_{1}}H^{n-p}_{\delta_{2}}\sqrt\mu), \omega_{t} H^{m}_{\delta_{1}}H^{n}_{\delta_{2}}\partial^{\alpha}_{x}f\right)_{L^{2}_{x, v}}dx,
\end{align*}      
then follows immediately from Lemma \ref{commutator-nonlinear} that
\begin{align*}     
     \left|\left(\omega_{t} H^{m}_{\delta_{1}}H^{n}_{\delta_{2}}\mathcal L_{2}f, \omega_{t} H^{m}_{\delta_{1}}H^{n}_{\delta_{2}}f\right)_{H^{3}_{x}L^{2}_{v}}\right|
     &\le C_{4}\sum_{l=0}^{m}\sum_{p=0}^{n}\binom{m}{l}\binom{n}{p}\sum_{|\alpha|\le3}\int_{\mathbb T^{3}_{x}}\|H^{l}_{\delta_{1}}H^{p}_{\delta_{2}}\partial^{\alpha}_{x}f(x, \cdot)\|_{L^{2}_{v}}\\
     &\quad\times\|\omega_{t} H^{m-l}_{\delta_{1}}H^{n-p}_{\delta_{2}}\sqrt\mu\|_{\sigma}\|\omega_{t} H^{m}_{\delta_{1}}H^{n}_{\delta_{2}}\partial^{\alpha}_{x}f(x, \cdot)\|_{\sigma}dx,
\end{align*} 
From Proposition 2.3 of~\cite{C-X-X}, there exists a positive constant $C_{0}$, depends on $\gamma, \ b$ and $c_{0}$ such that
\begin{align*}     
     \|\omega_{t} H^{m-l}_{\delta_{1}}H^{n-p}_{\delta_{2}}\sqrt\mu\|_{\sigma}\le\left(t^{\delta_{1}}\sqrt{C_{0}}\right)^{m-l}\left(t^{\delta_{2}}\sqrt{C_{0}}\right)^{n-p}\sqrt{(m-l+n-p+3)!}.
\end{align*}
Therefore, by using the Cauchy-Schwarz inequality, we have
\begin{align*}     
     &\left|\left(\omega_{t} H^{m}_{\delta_{1}}H^{n}_{\delta_{2}}\mathcal L_{2}f, \omega_{t} H^{m}_{\delta_{1}}H^{n}_{\delta_{2}}f\right)_{H^{3}_{x}L^{2}_{v}}\right|\le C_{5}\sum_{l=0}^{m}\sum_{p=0}^{n}\binom{m}{l}\binom{n}{p}\left(t^{\delta_{1}}\sqrt{C_{0}}\right)^{m-l}\left(t^{\delta_{2}}\sqrt{C_{0}}\right)^{n-p}\\
     &\qquad\qquad\qquad\times\sqrt{(m-l+n-p+3)!}\|H^{l}_{\delta_{1}}H^{p}_{\delta_{2}}f\|_{H^{3}_{x}L^{2}_{v}}|||\omega_{t} H^{m}_{\delta_{1}}H^{n}_{\delta_{2}}f|||.
\end{align*} 
\end{proof}


Next, we will prove the following upper bound for the operator $\mathcal L_{1}$.
\begin{prop}\label{H-m-L-1}
For any $m, n\in\mathbb N_{+}$, let $-3<\gamma<0$, then there exists a constant $C_{6}>0$, independent of $m$ and $n$, such that for all $\delta_{1}, \delta_{2}>2\max\{1, \frac{b-\gamma}{2b}\}$ and any suitable function $f$
\begin{align*}
     &\left|\left([\omega_{t} H^{m}_{\delta_{1}}H^{n}_{\delta_{2}}, \mathcal L_{1}]f, \omega_{t} H^{m}_{\delta_{1}}H^{n}_{\delta_{2}}f\right)_{H^{3}_{x}L^{2}_{v}}\right|
     \le\frac18|||\omega_{t} H^{m}_{\delta_{1}}H^{n}_{\delta_{2}}f|||^{2}+C_{6}\|\langle v\rangle^{\frac\gamma2}\omega_{t} H^{m}_{\delta_{1}}H^{n}_{\delta_{2}}f\|_{H^{3}_{x}L^{2}_{v}}\\
     &\qquad+C_{6}\sum_{l=1}^{m}\binom{m}{l}t^{\delta_{1}l}\sqrt{(l+1)!}|||\omega_{t} H^{m-l}_{\delta_{1}}H^{n}_{\delta_{2}}f|||\cdot|||\omega_{t} H^{m}_{\delta_{1}}H^{n}_{\delta_{2}}f|||\\
     &\qquad+C_{6}\sum_{p=1}^{n}\binom{n}{p}t^{\delta_{2}p}\sqrt{(p+1)!}|||\omega_{t} H^{m}_{\delta_{1}}H^{n-p}_{\delta_{2}}f|||\cdot|||\omega_{t} H^{m}_{\delta_{1}}H^{n}_{\delta_{2}}f|||\\
     &\qquad+C_{6}\sum_{l=1}^{m}\sum_{p=1}^{n}\binom{m}{l}\binom{n}{p}t^{\delta_{1}l}t^{\delta_{2}p}\sqrt{(l+p+1)!}|||\omega_{t} H^{m-l}_{\delta_{1}}H^{n-p}_{\delta_{2}}f|||\cdot|||\omega_{t} H^{m}_{\delta_{1}}H^{n}_{\delta_{2}}f|||.
\end{align*}
\end{prop}
\begin{proof}
     Let $F_{m, n}=\omega_{t} H^{m}_{\delta_{1}}H^{n}_{\delta_{2}}f$, from the representation for $\mathcal L_{1}$ in Lemma \ref{representations} and the fact $[H_{\delta_{l}}, A_{+, j}]=0$, it follows that
\begin{align*}
     \left([\omega_{t} H^{m}_{\delta_{1}}H^{n}_{\delta_{2}}, \mathcal L_{1}]f, F_{m, n}\right)_{H^{3}_{x}L^{2}_{v}}
     &=\sum_{|\alpha|\le3}\left(A_{+, j}\left(\omega_{t}H^{m}_{\delta_{1}}H^{n}_{\delta_{2}}\left((a_{jk}*\mu)A_{-, k}\partial^{\alpha}_{x}f\right)\right), \partial^{\alpha}_{x}F_{m, n}\right)_{L^{2}_{x, v}}\\
     &\quad-\sum_{|\alpha|\le3}\left(A_{+, j}\left((a_{jk}*\mu)\omega_{t}H^{m}_{\delta_{1}}H^{n}_{\delta_{2}}A_{-, k}\partial^{\alpha}_{x}f\right), \partial^{\alpha}_{x}F_{m, n}\right)_{L^{2}_{x, v}}\\
     &\quad+\sum_{|\alpha|\le3}\left([\omega_{t}, A_{+, j}]H^{m}_{\delta_{1}}H^{n}_{\delta_{2}}\left((a_{jk}*\mu)A_{-, k}\partial^{\alpha}_{x}f\right), \partial^{\alpha}_{x}F_{m, n}\right)_{L^{2}_{x, v}},
\end{align*}
then applying integration by parts, and Lemma \ref{Leibniz-type formula}, one can obtain that
\begin{align*}
     &\left([\omega_{t} H^{m}_{\delta_{1}}H^{n}_{\delta_{2}}, \mathcal L_{1}]f, F_{m, n}\right)_{H^{3}_{x}L^{2}_{v}}\\
     &=\sum_{|\alpha|\le3}\sum_{l=1}^{m}\sum_{p=0}^{n}\binom{m}{l}\binom{n}{p}\left(a_{jk}*\left(\sqrt\mu H^{l}_{\delta_{1}}H^{p}_{\delta_{2}}\sqrt\mu\right)\omega_{t}H^{m-l}_{\delta_{1}}H^{n-p}_{\delta_{2}}A_{-, k}\partial^{\alpha}_{x}f, A_{-, j}\partial^{\alpha}_{x}F_{m, n}\right)_{L^{2}_{x, v}}\\
     &\quad+\sum_{|\alpha|\le3}\sum_{p=1}^{n}\binom{n}{p}\left(a_{jk}*\left(\sqrt\mu H^{p}_{\delta_{2}}\sqrt\mu\right)\omega_{t}H^{m}_{\delta_{1}}H^{n-p}_{\delta_{2}}A_{-, k}\partial^{\alpha}_{x}f, A_{-, j}\partial^{\alpha}_{x}F_{m, n}\right)_{L^{2}_{x, v}}\\
     &\quad+\sum_{|\alpha|\le3}\left([\omega_{t}, A_{+, j}]H^{m}_{\delta_{1}}H^{n}_{\delta_{2}}\left((a_{jk}*\mu)A_{-, k}\partial^{\alpha}_{x}f\right), \partial^{\alpha}_{x}F_{m, n}\right)_{L^{2}_{x, v}}=Q_{1}+Q_{2}+Q_{3}.
\end{align*}
Now, we will show that 
\begin{align}\label{diff}
     \sqrt\mu H^{l}_{\delta_{1}}H^{p}_{\delta_{2}}\sqrt\mu= \left(-t^{\delta_{1}}\right)^{l}\left(-t^{\delta_{2}}\right)^{p}\sqrt\mu A_{+, 1}^{p+l}\sqrt\mu=t^{\delta_{1}l}t^{\delta_{2}p}\partial^{l+p}_{v_{1}}\mu, \ \forall \ l+p\ge1.
\end{align}
For the case of $p+l=1$, without loss of generality, we assume $l=1$, then
\begin{align*}
     \sqrt\mu H_{\delta_{1}}\sqrt\mu=-t^{\delta_{1}}\sqrt\mu A_{+, 1}\sqrt\mu=-t^{\delta_{1}}\sqrt\mu\left(\frac{v_{1}}{2}\sqrt\mu-\partial_{v_{1}}\sqrt\mu\right)=t^{\delta_{1}}\partial_{v_{1}}\mu.
\end{align*}
Assume that \eqref{diff} holds for $l+p-1$, then for the case of $l+p$, we have
\begin{align*}
     \sqrt\mu H_{\delta_{1}}\left(H^{l-1}_{\delta_{1}}H^{p}_{\delta_{2}}\sqrt\mu\right)&=H_{\delta_{1}}\left(\sqrt\mu H^{l-1}_{\delta_{1}}H^{p}_{\delta_{2}}\sqrt\mu\right)+\left[\sqrt\mu, H_{\delta_{1}}\right]\left(H^{l-1}_{\delta_{1}}H^{p}_{\delta_{2}}\sqrt\mu\right)\\
     &=t^{\delta_{1}}\partial_{v_{1}}\left(\sqrt\mu H^{l-1}_{\delta_{1}}H^{p}_{\delta_{2}}\sqrt\mu\right)=t^{\delta_{1}l}t^{\delta_{2}p}\partial^{l+p}_{v_{1}}\mu.
\end{align*}
Noting that 
\begin{align}\label{A-omega}
     [\omega_{t}, A_{\pm, k}]=\pm\partial_{k}\omega_{t}=\frac{\pm c_{0}b}{1+t}\langle v\rangle^{b-2}v_{k}\omega_{t}, \quad 1\le k\le3,
\end{align}
then we have
\begin{align*}
     &Q_{1}+Q_{2}=\sum_{|\alpha|\le3}\sum_{l=1}^{m}\binom{m}{l}t^{\delta_{1}l}\left(\partial^{l}_{v_{1}}\bar a_{jk}A_{-, k}\partial^{\alpha}_{x}F_{m-l, n}, \  A_{-, j}\partial^{\alpha}_{x}F_{m, n}\right)_{L^{2}_{x, v}}\\
     &\quad-\frac{c_{0}b}{1+t}\sum_{|\alpha|\le3}\sum_{l=0}^{m}\binom{m}{l}t^{\delta_{1}l}\left(a_{jk}*(v_{k}\partial^{l}_{v_{1}}\mu)\langle v\rangle^{b-2}\partial^{\alpha}_{x}F_{m-l, n}, \ A_{-, j}\partial^{\alpha}_{x}F_{m, n}\right)_{L^{2}_{x, v}}\\
     &\quad+\sum_{|\alpha|\le3}\sum_{l=0}^{m}\binom{m}{l}t^{\delta_{1}l}\left(\partial^{l}_{v_{1}}\bar a_{jk}\omega_{t}[H^{m-l}_{\delta_{1}}H^{n}_{\delta_{2}}, A_{-, k}]\partial^{\alpha}_{x}f, \  A_{-, j}\partial^{\alpha}_{x}F_{m, n}\right)_{L^{2}_{x, v}}\\
     &\quad+\sum_{|\alpha|\le3}\sum_{p=1}^{n}\binom{n}{p}t^{\delta_{2}p}\left(\partial^{p}_{v_{1}}\bar a_{jk}A_{-, k}\partial^{\alpha}_{x}F_{m, n-p}, \ A_{-, j}\partial^{\alpha}_{x}F_{m, n}\right)_{L^{2}_{x, v}}\\
     &\quad-\frac{c_{0}b}{1+t}\sum_{|\alpha|\le3}\sum_{p=1}^{n}\binom{n}{p}t^{\delta_{2}p}\left(a_{jk}*\left(v_{k}\partial^{p}_{v_{1}}\mu\right)\langle v\rangle^{b-2}\partial^{\alpha}_{x}F_{m, n-p}, \ A_{-, j}\partial^{\alpha}_{x}F_{m, n}\right)_{L^{2}_{x, v}}\\
     &\quad+\sum_{|\alpha|\le3}\sum_{p=1}^{n}\binom{n}{p}t^{\delta_{2}p}\left(\partial^{p}_{v_{1}}\bar a_{jk}\omega_{t} \left[H^{m}_{\delta_{1}}H^{n-p}_{\delta_{2}}, A_{-, k}\right]\partial^{\alpha}_{x}f, \ A_{-, j}\partial^{\alpha}_{x}F_{m, n}\right)_{L^{2}_{x, v}}\\
     &\quad+\sum_{|\alpha|\le3}\sum_{l=1}^{m}\sum_{p=1}^{n}\binom{m}{l}\binom{n}{p}t^{\delta_{1}l}t^{\delta_{2}p}\left(\partial^{p+l}_{v_{1}}\bar a_{jk}A_{-, k}\partial^{\alpha}_{x}F_{m-l, n-p}, \ A_{-, j}\partial^{\alpha}_{x}F_{m, n}\right)_{L^{2}_{x, v}}\\
     &\quad-\frac{c_{0}b}{1+t}\sum_{|\alpha|\le3}\sum_{l=1}^{m}\sum_{p=1}^{n}\binom{m}{l}\binom{n}{p}t^{\delta_{1}l}t^{\delta_{2}p}\left(a_{jk}*\left(v_{k}\partial^{p+l}_{v_{1}}\mu\right)\langle v\rangle^{b-2}\partial^{\alpha}_{x}F_{m-l, n-p}, \ A_{-, j}\partial^{\alpha}_{x}F_{m, n}\right)_{L^{2}_{x, v}}\\
     &\quad+\sum_{|\alpha|\le3}\sum_{l=1}^{m}\sum_{p=1}^{n}\binom{m}{l}\binom{n}{p}t^{\delta_{1}l}t^{\delta_{2}p}\left(\partial^{p+l}_{v_{1}}\bar a_{jk}\omega_{t}\left[H^{m-l}_{\delta_{1}}H^{n-p}_{\delta_{2}}, A_{-, k}\right]\partial^{\alpha}_{x}f, \ A_{-, j}\partial^{\alpha}_{x}F_{m, n}\right)_{L^{2}_{x, v}}\\
     &=I_{1}+I_{2}+I_{3}+I_{4}+I_{5}+I_{6}+I_{7}+I_{8}+I_{9},
\end{align*}
where $\bar a_{jk}=a_{jk}*\mu$. For the term $I_{1}$, we can write it as 
\begin{align*}
     I_{1}&=\sum_{|\alpha|\le3}mt^{\delta_{1}}\left(\partial_{v_{1}}\bar a_{jk}A_{-, k}\partial^{\alpha}_{x}F_{m-1, n}, \ A_{-, j}\partial^{\alpha}_{x}F_{m, n}\right)_{L^{2}_{x, v}}\\
     &\quad+\sum_{|\alpha|\le3}\sum_{l=2}^{m}\binom{m}{l}t^{\delta_{1}l}\left(\partial^{l}_{v_{1}}\bar a_{jk}A_{-, k}\partial^{\alpha}_{x}F_{m-l, n}, \ A_{-, j}\partial^{\alpha}_{x}F_{m, n}\right)_{L^{2}_{x, v}}=I_{1, 1}+I_{1, 2}.
\end{align*}     

For the term $I_{1, 1}$, decomposing $\mathbb R^{3}\times\mathbb R^{3}=\left\{|v|\le1\right\}\cup\left\{2|v'|\ge|v|, |v|\ge1\right\}\cup\left\{2|v'|\le|v|, |v|\ge1\right\}=\Omega_{1}\cup\Omega_{2}\cup\Omega_{3}$.
In $\Omega_{1}\cup\Omega_{2}$, noting that $|\partial_{v_{1}} a_{jk}|\lesssim|v|^{\gamma+1}$, then follows immediately from \eqref{singular integration} that
$$|\partial_{v_{1}}\bar a_{jk}|=|\partial_{v_{1}} a_{jk}*\mu|\lesssim\langle v\rangle^{\gamma},$$
by using \eqref{upper bound1} and the Cauchy-Schwarz inequality, we have
\begin{align*}
     \left|I_{1, 1}\big|_{\Omega_{1}\cup\Omega_{2}}\right|&\lesssim mt^{\delta_{1}}|||F_{m-1, n}|||\cdot|||F_{m, n}|||.
\end{align*}
In $\Omega_{3}$, using Taylor's expansion
$$a_{jk}(v-v')=a_{jk}(v)+\sum_{l=1}^{3}\int_{0}^{1}\partial_{l}a_{jk}(v-sv')dsv'_{l},$$
since
$$\sum_{j}a_{jk}v_{j}=\sum_{k}a_{jk}v_{k}=0,$$
we can obtain that 
\begin{align*}
     &I_{1, 1}\big|_{\Omega_{3}}=\sum_{|\alpha|\le3}mt^{\delta_{1}}\int_{\Omega_{3}\times\mathbb T^{3}_{x}}\partial_{1}A(v)\mu(v')\Big[\left(\mathbf I-\mathbf P_{v}\right)\nabla_{\mathcal H_{-}}\partial^{\alpha}_{x}F_{m-1, n}\left(\mathbf I-\mathbf P_{v}\right)\nabla_{\mathcal H_{-}}\partial^{\alpha}_{x}F_{m, n}\\
     &\quad+\mathbf P_{v}\nabla_{\mathcal H_{-}}\partial^{\alpha}_{x}F_{m-1, n}\left(\mathbf I-\mathbf P_{v}\right)\nabla_{\mathcal H_{-}}\partial^{\alpha}_{x}F_{m, n}+\left(\mathbf I-\mathbf P_{v}\right)\nabla_{\mathcal H_{-}}\partial^{\alpha}_{x}F_{m-1, n}\mathbf P_{v}\nabla_{\mathcal H_{-}}\partial^{\alpha}_{x}F_{m, n}\Big]\\
     &\quad+\sum_{|\alpha|\le3}mt^{\delta_{1}}\sum_{l=1}^{3}\int_{\Omega_{3}\times\mathbb T^{3}_{x}}\int_{0}^{1}\partial_{1l}a_{jk}(v-sv')dsv'_{l}A_{-, k}\partial^{\alpha}_{x}F_{m-1, n}A_{-, j}\partial^{\alpha}_{x}F_{m, n},
\end{align*}
since $|\partial_{1}a_{jk}(v)|\lesssim\langle v\rangle^{\gamma+1}$ and $|\partial_{l}\partial_{1}a_{jk}(v-sv')|\lesssim\langle v\rangle^{\gamma}$ for all $(v', v)\in\Omega_{3}$, then using \eqref{upper bound1} and the Cauchy-Schwarz inequality, we have
\begin{align*}
     \left|I_{1, 1}\big|_{\Omega_{3}}\right|&\lesssim mt^{\delta_{1}}|||F_{m-1, n}|||\cdot|||F_{m, n}|||.
\end{align*}
An argument similar to the one used in the Lemma 2.1 of~\cite{LX2} shows that
\begin{equation}\label{derivation a}
   |\partial^{l}_{v_{1}}\bar a_{jk}|\lesssim \langle v\rangle^{\gamma}\sqrt{l!}, \quad \forall \ l\ge2,
\end{equation}
thus, applying \eqref{upper bound1} and Cauchy-Schwarz inequality 
\begin{align*}
     \left|I_{1, 2}\right|&\lesssim\sum_{l=2}^{m}\binom{m}{l}t^{\delta_{1}l}\sqrt{l!}|||F_{m-l, n}|||\cdot|||F_{m, n}|||.
\end{align*}
For the term $I_{2}$, noting that $[v_{k}, \partial_{v_{1}}]=\delta_{1k}$, then we can write is as
\begin{align*}
     I_{2}&=-\frac{c_{0}b}{1+t}\sum_{|\alpha|\le3}\left(a_{jk}*(v_{k}\mu)\langle v\rangle^{b-2}\partial^{\alpha}_{x}F_{m, n}, \ A_{-, j}\partial^{\alpha}_{x}F_{m, n}\right)_{L^{2}_{x, v}}\\
     &\quad-\frac{c_{0}b}{1+t}\sum_{|\alpha|\le3}mt^{\delta_{1}}\left(a_{jk}*(\delta_{1k}\mu)\langle v\rangle^{b-2}\partial^{\alpha}_{x}F_{m-1, n}, \ A_{-, j}\partial^{\alpha}_{x}F_{m, n}\right)_{L^{2}_{x, v}}\\
     &\quad-\frac{c_{0}b}{1+t}\sum_{|\alpha|\le3}\sum_{l=1}^{m}\binom{m}{l}t^{\delta_{1}l}\left(\partial^{l}_{v_{1}}a_{jk}*(v_{k}\mu)\langle v\rangle^{b-2}\partial^{\alpha}_{x}F_{m-l, n}, \ A_{-, j}\partial^{\alpha}_{x}F_{m, n}\right)_{L^{2}_{x, v}}\\
     &\quad-\frac{c_{0}b}{1+t}\sum_{|\alpha|\le3}\sum_{l=2}^{m}\binom{m}{l}t^{\delta_{1}l}\left(\partial^{l-1}_{v_{1}}a_{jk}*(\delta_{1k}\mu)\langle v\rangle^{b-2}\partial^{\alpha}_{x}F_{m-l, n}, \ A_{-, j}\partial^{\alpha}_{x}F_{m, n}\right)_{L^{2}_{x, v}}.
\end{align*}
Noting that $0< b\le 2$, we discuss it as $I_{1, 1}$, then the first two terms can be bounded by 
$$\|\langle v\rangle^{\gamma/2}F_{m, n}\|_{H^{3}_{x}L^{2}_{v}}|||F_{m, n}|||\quad {\rm and} \quad mt^{\delta_{1}}|||F_{m-1, n}|||\cdot|||F_{m, n}|||.$$
To bound the other terms, we use \eqref{upper bound1}, \eqref{derivation a} and the Cauchy-Schwarz inequality, it can be bounded by
$$\sum_{l=1}^{m}\binom{m}{l}t^{\delta_{1}l}\sqrt{l!}|||F_{m-l, n}|||\cdot|||F_{m, n}|||.$$
Thus, we have
\begin{align*}
     \left|I_{2}\right|&\lesssim\|\langle v\rangle^{\gamma/2}F_{m, n}\|_{H^{3}_{x}L^{2}_{v}}|||F_{m, n}|||+\sum_{l=1}^{m}\binom{m}{l}t^{\delta_{1}l}\sqrt{l!}|||F_{m-l, n}|||\cdot|||F_{m, n}|||.
\end{align*}
For the term $I_{3}$, since
\begin{align*}
     &[H_{\delta_{j}}, A_{-, k}]=-t^{\delta}[A_{+, 1}, A_{-, k}]=0, \ (k\ne1), \ j=1, 2,\\
     &[H_{\delta_{j}}, A_{-, 1}]=-t^{\delta}[A_{+, 1}, A_{-, 1}]=t^{\delta_{j}}, \ j=1, 2,
\end{align*}
one can deduce that $H^{m}_{\delta_{1}}H^{n}_{\delta_{2}}A_{-, k}=A_{-, k}H^{m}_{\delta_{1}}H^{n}_{\delta_{2}}$ for $k\ne1$ and
$$H^{m}_{\delta_{1}}H^{n}_{\delta_{2}}A_{-, 1}=A_{-, 1}H^{m}_{\delta_{1}}H^{n}_{\delta_{2}}+mt^{\delta_{1}}H^{m-1}_{\delta_{1}}H^{n}_{\delta_{2}}+nt^{\delta_{2}}H^{m}_{\delta_{1}}H^{n-1}_{\delta_{2}},$$
these lead to
\begin{align*}
     I_{3}&=t^{\delta_{1}}\sum_{|\alpha|\le3}\sum_{l=0}^{m-1}\binom{m}{l}(m-l)t^{\delta_{1}l}\left(\partial^{l}_{v_{1}}\bar a_{jk}\delta_{1k}\partial^{\alpha}_{x}F_{m-l-1, n}, \ A_{-, j}\partial^{\alpha}_{x}F_{m, n}\right)_{L^{2}_{x, v}}\\
     &\quad+nt^{\delta_{2}}\sum_{|\alpha|\le3}\sum_{l=0}^{m}\binom{m}{l}t^{\delta_{1}l}\left(\partial^{l}_{v_{1}}\bar a_{jk}\delta_{1k}\partial^{\alpha}_{x}F_{m-l, n-1}, A_{-, j}\partial^{\alpha}_{x}F_{m, n}\right)_{L^{2}_{x, v}}.
\end{align*}
If $l=0$, we discuss it as $I_{1, 2}$, then it can be bounded by 
\begin{align*}
     mt^{\delta_{1}}|||F_{m-1, n}|||\cdot|||F_{m, n}|||+nt^{\delta_{2}}|||F_{m, n-1}|||\cdot|||F_{m, n}|||.
\end{align*}
If $l\ge1$, by using \eqref{upper bound1}, \eqref{derivation a} and the Cauchy-Schwarz, one can get
\begin{align*}
     &\left|t^{\delta_{1}}\sum_{|\alpha|\le3}\sum_{l=1}^{m-1}\binom{m}{l}(m-l)t^{\delta_{1}l}\left(\partial^{l}_{v_{1}}\bar a_{jk}\delta_{1k}\partial^{\alpha}_{x}F_{m-l-1, n}, \ A_{-, j}\partial^{\alpha}_{x}F_{m, n}\right)_{L^{2}_{x, v}}\right|\\
     &\quad+\left|nt^{\delta_{2}}\sum_{|\alpha|\le3}\sum_{l=1}^{m}\binom{m}{l}t^{\delta_{1}l}\left(\partial^{l}_{v_{1}}\bar a_{jk}\delta_{1k}\partial^{\alpha}_{x}F_{m-l, n-1}, \ A_{-, j}\partial^{\alpha}_{x}F_{m, n}\right)_{L^{2}_{x, v}}\right|\\
     &\lesssim \sum_{l=1}^{m-1}\binom{m}{l}t^{\delta_{1}l}\sqrt{l!}(m-l)t^{\delta_{1}}|||F_{m-l-1, n}|||\cdot|||F_{m, n}|||+\sum_{l=1}^{m}\binom{m}{l}t^{\delta_{1}l}\sqrt{l!}nt^{\delta_{2}}|||F_{m-l, n-1}|||\cdot|||F_{m, n}|||.
\end{align*}
Similarly, we can deduce that 
\begin{align*}
     \left|I_{4}+I_{5}+I_{6}\right|&\lesssim\sum_{p=1}^{n}\binom{n}{p}t^{\delta_{2}p}\sqrt{p!}|||F_{m, n-p}|||\cdot|||F_{m, n}|||+\sum_{p=1}^{n}\binom{n}{p}t^{\delta_{2}p}\sqrt{p!}mt^{\delta_{1}}|||F_{m-1, n-p}|||\cdot|||F_{m, n}|||\\
     &\quad+\sum_{p=1}^{n-1}\binom{n}{p}t^{\delta_{2}p}\sqrt{p!}(n-p)t^{\delta_{2}}|||F_{m, n-p-1}|||\cdot|||F_{m, n}|||,
\end{align*}
and
\begin{align*}
     &\left|I_{7}+I_{8}+I_{9}\right|\lesssim\sum_{l=1}^{m}\sum_{p=1}^{n}\binom{m}{l}\binom{n}{p}t^{\delta_{1}l}t^{\delta_{2}p}\sqrt{(l+p)!}|||F_{m-l, n-p}|||\cdot|||F_{m, n}|||\\
     &\quad+\sum_{l=1}^{m-1}\sum_{p=1}^{n}\binom{m}{l}\binom{n}{p}t^{\delta_{1}l}t^{\delta_{2}p}\sqrt{(l+p)!}(m-l)t^{\delta_{1}}|||F_{m-l-1, n-p}|||\cdot|||F_{m, n}|||\\
     &\quad+\sum_{l=1}^{m}\sum_{p=1}^{n-1}\binom{m}{l}\binom{n}{p}t^{\delta_{1}l}t^{\delta_{2}p}\sqrt{(l+p)!}(n-p)t^{\delta_{2}}|||F_{m-l, n-p-1}|||\cdot|||F_{m, n}|||.
\end{align*}
Next, we consider the term $Q_{3}$. Applying \eqref{diff} and \eqref{A-omega}, we can write it as
\begin{align*}
     Q_{3}&=\frac{c_{0}b}{1+t}\sum_{|\alpha|\le3}\sum_{l=0}^{m}\sum_{p=0}^{n}\binom{m}{l}\binom{n}{p}t^{\delta_{1}l}t^{\delta_{2}p}\left(a_{jk}*\left(v_{k} \partial_{v_{1}}^{l+p}\mu\right)\langle v\rangle^{b-2}A_{-, k}\partial^{\alpha}_{x}F_{m-l, n-p}, \ \partial^{\alpha}_{x}F_{m, n}\right)_{L^{2}_{x, v}}\\
     &\quad-\left(\frac{c_{0}b}{1+t}\right)^{2}\sum_{|\alpha|\le3}\sum_{l=0}^{m}\sum_{p=0}^{n}\binom{m}{l}\binom{n}{p}t^{\delta_{1}l}t^{\delta_{2}p}\left(a_{jk}*\left(v_{k}^{2} \partial_{v_{1}}^{l+p}\mu\right)\langle v\rangle^{2(b-2)}\partial^{\alpha}_{x}F_{m-l, n-p}, \ \partial^{\alpha}_{x}F_{m, n}\right)_{L^{2}_{x, v}}\\
     &\quad+\frac{c_{0}b}{1+t}\sum_{|\alpha|\le3}\sum_{l=0}^{m}\sum_{p=0}^{n}\binom{m}{l}\binom{n}{p}t^{\delta_{1}l}t^{\delta_{2}p}\left(a_{jk}*\left(v_{k} \partial_{v_{1}}^{l+p}\mu\right)\langle v\rangle^{b-2}\partial^{\alpha}_{x}F_{m-l, n-p}, \ \partial^{\alpha}_{x}F_{m, n}\right)_{L^{2}_{x, v}}.
\end{align*}
Since $c_{0}$ small, by the same technique, we can also prove that
\begin{align*}
     \left|Q_{3}\right|&\le\frac{1}{16}|||F_{m, n}|||^{2}
     +\tilde C_{6}\sum_{l=1}^{m}\binom{m}{l}t^{\delta_{1}l}\sqrt{l!}|||F_{m-l, n}|||\cdot|||F_{m, n}|||\\
     &\quad+\tilde C_{6}\sum_{p=1}^{n}\binom{n}{p}t^{\delta_{2}p}\sqrt{p!}|||F_{m, n-p}|||\cdot|||F_{m, n}|||\\
     &\quad+\tilde C_{6}\sum_{l=1}^{m}\sum_{p=1}^{n}\binom{m}{l}\binom{n}{p}t^{\delta_{1}l}t^{\delta_{2}p}\sqrt{(l+p)!}|||F_{m-l, n-p}|||\cdot|||F_{m, n}|||\\
     &\quad+\tilde C_{6}\sum_{l=0}^{m-1}\sum_{p=0}^{n}\binom{m}{l}\binom{n}{p}t^{\delta_{1}l}t^{\delta_{2}p}\sqrt{(l+p)!}(m-l)t^{\delta_{1}}|||F_{m-l-1, n-p}|||\cdot|||F_{m, n}||\\
     &\quad+\tilde C_{6}\sum_{l=0}^{m}\sum_{p=0}^{n-1}\binom{m}{l}\binom{n}{p}t^{\delta_{1}l}t^{\delta_{2}p}\sqrt{(l+p)!}(n-p)t^{\delta_{2}}|||F_{m-l, n-p-1}|||\cdot|||F_{m, n}|||.
\end{align*}
Using the change of variables $l+1\to l$ and $p+1\to p$, we have
\begin{align*}
     &\left|Q_{3}\right|\le\frac{1}{16}|||F_{m, n}|||^{2}+\tilde C_{6}\sum_{l=1}^{m}\sum_{p=1}^{n}\binom{m}{l}\binom{n}{p}t^{\delta_{1}l}t^{\delta_{2}p}\sqrt{(l+p+1)!}|||F_{m-l, n-p}|||\cdot|||F_{m, n}|||\\
     &\quad+\tilde C_{6}\sum_{l=1}^{m}\binom{m}{l}t^{\delta_{1}l}\sqrt{(l+1)!}|||F_{m-l, n}|||\cdot|||F_{m, n}|||+\tilde C_{6}\sum_{p=1}^{n}\binom{n}{p}t^{\delta_{2}p}\sqrt{(p+1)!}|||F_{m, n-p}|||\cdot|||F_{m, n}|||.
\end{align*}
Combining these results then follows from the Cauchy-Schwarz inequality that there exists a positive constant $C_{6}$, independent of $m$ and $n$, such that
\begin{align*}
     &\left|\left([\omega H^{m}_{\delta_{1}}H^{n}_{\delta_{2}}, \mathcal L_{1}]f, \omega_{t} H^{m}_{\delta_{1}}H^{n}_{\delta_{2}}f\right)_{H^{3}_{x}L^{2}_{v}}\right|
     \le\frac18|||F_{m, n}|||^{2}+C_{6}\|\langle v\rangle^{\frac\gamma2}F_{m, n}\|_{H^{3}_{x}L^{2}_{v}}\\
     &\quad+C_{6}\sum_{l=1}^{m}\binom{m}{l}t^{\delta_{1}l}\sqrt{(l+1)!}|||F_{m-l, n}|||\cdot|||F_{m, n}|||+C_{6}\sum_{p=1}^{n}\binom{n}{p}t^{\delta_{2}p}\sqrt{(p+1)!}|||F_{m, n-p}|||\cdot|||F_{m, n}|||\\
     &\quad+C_{6}\sum_{l=1}^{m}\sum_{p=1}^{n}\binom{m}{l}\binom{n}{p}t^{\delta_{1}l}t^{\delta_{2}p}\sqrt{(l+p+1)!}|||F_{m-l, n-p}|||\cdot|||F_{m, n}|||.
\end{align*}
\end{proof}

\section{Energy estimates for one directional derivations}\label{s5}

This section aims to establish the energy estimates for one-directional derivation. First, we consider the energy estimates of the solution.
\begin{lemma}\label{lemma4.1}
    Let $f$ be a solution of \eqref{1-2} with $\|f\|_{L^{\infty}([0, T]; H^3_xL^{2}_v(\omega_{t}))}$ small enough. Then we have
\begin{align}\label{m=0}
      \|\omega_{t} f(t)\|^2_{H^3_xL^{2}_v}+\frac{2c_{0}}{(1+T)^{2}}\int_{0}^{t}\|\langle v\rangle^{\frac b2}\omega_{\tau} f(\tau)\|^{2}_{H^{3}_{x}L^{2}_{v}}d\tau+\int_0^t|||\omega_{\tau} f(\tau)|||^2 d\tau\le (B\epsilon)^{2}, \quad \forall \ 0<t\le T,
\end{align}
with $B>0$ depends on $\gamma, \ b, \ c_{0}$ and $T$.
\end{lemma}
\begin{proof}
     Since $f$ is the solution of Cauchy problem \eqref{1-2}
one can get that
     $$\frac{1}{2}\frac{d}{dt}\|\omega_{t} f(t)\|^{2}_{H^{3}_{x}L^{2}_{v}}+\frac{c_{0}}{(1+t)^{2}}\|\langle v\rangle^{\frac b2}\omega_{t} f(t)\|^{2}_{H^{3}_{x}L^{2}_{v}}+(\omega_{t}\mathcal L_{1}f, \omega_{t} f)_{H^{3}_{x}L^{2}_{v}}=(\omega_{t}\Gamma(f, f), \omega_{t} f)_{H^{3}_{x}L^{2}_{v}}-(\omega_{t}\mathcal L_{2}f, \omega_{t} f)_{H^{3}_{x}L^{2}_{v}}.$$
Since $\mathcal L_{2}f=-\Gamma(f, \sqrt\mu)$ and $c_{0}$ small, noting that $\partial^{\alpha}_{x}\mathcal L_{2}f=\mathcal L_{2}\partial^{\alpha}_{x}f$, then from Lemma \ref{commutator-nonlinear}, one has
\begin{align*}
     &\left|(\omega_{t}\mathcal L_{2}f, \omega_{t} f)_{H^{3}_{x}L^{2}_{v}}\right|\le\tilde C_{3}\sum_{|\alpha|\le3}\int_{\mathbb T^{3}_{x}}\|\partial^{\alpha}_{x}f\|_{L^{2}_{v}}\|\partial^{\alpha}_{x}\omega_{t} f\|_{\sigma}dx\le\tilde C_{3}\|f\|_{H^{3}_{x}L^{2}_{v}}|||\omega_{t} f|||.
\end{align*}
Since $\gamma<0$, then follows immediately from Lemma \ref{commutator-nonlinear}, Proposition \ref{H-m-L-1} and \eqref{L1} that
\begin{align*}
     &\frac{1}{2}\frac{d}{dt}\|\omega_{t} f(t)\|^{2}_{H^{3}_{x}L^{2}_{v}}+\frac{c_{0}}{(1+t)^{2}}\|\langle v\rangle^{\frac b2}\omega_{t} f(t)\|^{2}_{H^{3}_{x}L^{2}_{v}}+\frac34|||\omega_{t} f(t)|||^{2}\\
     &\le C_{4}\left\|f(t)\right\|_{H^{3}_{x}L^{2}_{v}}|||\omega_{t} f(t)|||+\tilde C_{3}\|f(t)\|_{H^{3}_{x}L^{2}_{v}}|||\omega_{t} f(t)|||.
\end{align*}
By using the Cauchy-Schwarz inequality and the fact
$$\|\omega_{t} f\|_{L^{\infty}([0, T]; H^3_xL^{2}_v)}\le\epsilon, \quad \forall \ 0<\epsilon<1,$$
we can deduce that
\begin{align*}
     &\frac{1}{2}\frac{d}{dt}\|\omega_{t} f(t)\|^{2}_{H^{3}_{x}L^{2}_{v}}+\frac{c_{0}}{(1+t)^{2}}\|\langle v\rangle^{\frac b2}\omega_{t} f(t)\|^{2}_{H^{3}_{x}L^{2}_{v}}+\frac12|||\omega_{t} f(t)|||^{2}
     \le 2(\tilde C_{3})^{2}\|f(t)\|^{2}_{H^{3}_{x}L^{2}_{v}},
\end{align*}
if taking $C_{4}\epsilon\le\frac18$. Integrating from~0~ to $t$, it follows that
\begin{align}\label{integration}
\begin{split}
     &\|\omega_{t} f(t)\|^{2}_{H^{3}_{x}L^{2}_{v}}+\frac{2c_{0}}{(1+T)^{2}}\int_{0}^{t}\|\langle v\rangle^{\frac b2}\omega_{\tau} f(\tau)\|^{2}_{H^{3}_{x}L^{2}_{v}}d\tau+\int_{0}^{t}|||\omega_{\tau} f(\tau)|||^{2}d\tau\\
     &\le\|\omega_{0} f_{0}\|^{2}_{H^{3}_{x}L^{2}_{v}}+4\tilde C_{3}\int_{0}^{t}\|\omega_{\tau} f(\tau)\|^{2}_{H^{3}_{x}L^{2}_{x}}d\tau, \quad 0<t\le T,
\end{split}
\end{align}
 by Gronwall inequality, we get for all $0<t\le T$
$$\|\omega_{t} f(t)\|^{2}_{H^{3}_{x}L^{2}_{v}}\le\left(1+4T\tilde C_{3}e^{4T\tilde C_{3}}\right)\|\omega_{0} f_{0}\|^{2}_{H^{3}_{x}L^{2}_{v}},$$
plugging it back into \eqref{integration}, one has for all $0<t\le T$
\begin{align*}
     &\|\omega_{t} f(t)\|^{2}_{H^{3}_{x}L^{2}_{v}}+\frac{c_{0}}{(1+T)^{2}}\int_{0}^{t}+\|\langle v\rangle^{\frac b2}\omega_{\tau} f(\tau)\|^{2}_{H^{3}_{x}L^{2}_{v}}d\tau+\int_{0}^{t}|||\omega_{\tau} f(\tau)|||^{2}d\tau\le B^{2}\epsilon^{2},
\end{align*}
if taking $B\ge 1+4T\tilde C_{3}e^{4T\tilde C_{3}}$.
\end{proof}

Now, we turn to establish the energy estimates for one-directional derivation.
\begin{lemma}\label{lemma4.1-1}
     Let $f$ be the smooth solution of \eqref{1-2} with $\|f\|_{L^{\infty}([0, T]; H^3_xL^{2}_v(\omega_{t}))}$ small enough. Then for all $\delta_{1}, \delta_{2}$ satisfies \eqref{delta12}, there exists a constant $\tilde B>0$ such that for $j=1, 2$ and $0<t\le T$
    \begin{align}\label{k=1}
        \begin{split}
            &\|\omega_{t} H_{\delta_{j}}f(t)\|^2_{H^{3}_{x}L^{2}_{v}}+\frac{c_{0}}{(1+T)^{2}}\int_{0}^{t}\|\langle v\rangle^{\frac b2}\omega_{\tau} H_{\delta_{j}}f(\tau)\|^{2}_{H^{3}_{x}L^{2}_{v}}d\tau+\int_{0}^{t}|||\omega_{\tau} H_{\delta_{j}}f(\tau)|||^{2}d\tau\le \tilde B^{2}\epsilon^{2}.
        \end{split}
    \end{align} 
\end{lemma}

\begin{proof}
From \eqref{1-2}, \eqref{kehigher-1} and \eqref{A-omega}, we have
\begin{align}\label{dt-H}
\begin{split}
     &\frac12\frac{d}{dt}\|\omega_{t} H_{\delta_{j}}f(t)\|^2_{H^3_xL^{2}_v}+\frac{c_{0}}{(1+t)^{2}}\|\langle v\rangle^{\frac b2}\omega_{t}H_{\delta_{j}}f(t)\|^{2}_{H^{3}_{x}L^{2}_{v}}+(\omega_{t} H_{\delta_{j}}\mathcal L_1f,  \ \omega_{t} H_{\delta_{j}}f)_{H^3_xL^{2}_v}\\
     &=-\delta_{j} t^{\delta_{j}-1}(\omega_{t} A_{+, 1}f, \omega_{t} H_{\delta_{j}} f)_{H^3_xL^{2}_v}-\delta_{j} t^{\delta_{j}-1}( \partial_{v_{1}}\omega_{t} f, \omega_{t} H_{\delta_{j}} f)_{H^3_xL^{2}_v}\\
     &\quad-(\omega_{t} H_{\delta_{j}} \mathcal L_2f,  \omega_{t} H_{\delta_{j}}f)_{H^3_xL^{2}_v}+(\omega_{t} H_{\delta_{j}}\Gamma(f, f),  \omega_{t} H_{\delta_{j}}f)_{H^3_xL^{2}_v}.
\end{split}
\end{align}
    Since $\gamma<0$, then follows immediately from Proposition \ref{H-m-L-1} that
\begin{align*}
     (\omega_{t} H_{\delta_{j}}\mathcal L_{1}f,\ \omega_{t} H_{\delta_{j}}f)_{H^{3}_{x}L^{2}_{v}}&\ge\frac12|||\omega_{t} H_{\delta_{j}}f|||^{2}-C_{6}\left\|\omega_{t} H_{\delta_{j}}f\right\|^{2}_{H^{3}_{x}L^{2}_{v}}-(C_{6})^{2}t^{2\delta_{j}}|||\omega_{t} f|||^{2}.
\end{align*}
By using the Cauchy-Schwarz inequality and the fact $0< b\le2$, we have
\begin{align*}
     \delta_{j} t^{\delta_{j}-1}\left|( \partial_{v_{1}}\omega_{t} f, \omega_{t} H_{\delta_{j}} f)_{H^3_xL^{2}_v}\right|
     &\le \frac{c_{0}b}{1+t}\delta_{j} t^{\delta_{j}-1}\|\langle v\rangle^{\frac b2}\omega_{t} f\|_{H^{3}_{x}L^{2}_{v}}\|\langle v\rangle^{\frac b2}\omega_{t} H_{\delta_{j}}f\|_{H^{3}_{x}L^{2}_{v}}.
\end{align*}
Applying Proposition \ref{H-m-Gamma} with $n=0$, it follows that
\begin{align*}
     |(\omega_{t} H_{\delta_{j}}\Gamma(f, f),  \omega_{t} H_{\delta_{j}}f)_{H^3_xL^{2}_v}|
     &\le C_{4}\|f\|_{H^{3}_{x}L^{2}_{v}}|||\omega_{t} H_{\delta_{j}}f|||^{2}+C_{4}\|H_{\delta_{j}}f\|_{H^{3}_{x}L^{2}_{v}}|||\omega_{t} f|||\cdot|||\omega_{t} H_{\delta_{j}}f|||.
\end{align*}
By using Corollary \ref{L2} with $n=0$, it follows that
\begin{align*}
     &|(\omega_{t} H_{\delta_{j}}\mathcal L_{2}f,\ \omega_{t} H_{\delta_{j}}f)_{H^{3}_{x}L^{2}_{v}}|\le C_{5}\sqrt{C_{0}}t^{\delta_{j}}\|f\|_{H^{3}_{x}L^{2}_{v}}|||\omega_{t} H_{\delta_{j}}f|||+C_{5}\|H_{\delta_{j}}f\|_{H^{3}_{x}L^{2}_{v}}|||\omega_{t} H_{\delta_{j}}f|||.
\end{align*}  
For the first term on the right-hand side of \eqref{dt-H}, by using the Cauchy-Schwarz inequality, we have
\begin{align*}
     \left|\delta_{j} t^{\delta_{j}-1}(\omega_{t} A_{+, 1}f, \omega_{t} H_{\delta_{j}} f)_{H^3_xL^{2}_v}\right|&\le\delta_{j} t^{\delta_{j}-1}\left\|\omega_{t} A_{+, 1}f\right\|_{H^{3}_{x}L^{2}_{v}}\left\|\omega_{t} H_{\delta_{j}}f\right\|_{H^{3}_{x}L^{2}_{v}}\\
     &\le\varepsilon\left\|\omega_{t} H_{\delta_{j}}f\right\|_{H^{3}_{x}L^{2}_{v}}^{2}+\varepsilon^{-1}\delta_{j}^{2} t^{2(\delta_{j}-1)}\left\|\omega_{t} A_{+, 1}f\right\|_{H^{3}_{x}L^{2}_{v}}^{2},
\end{align*}
since $0<\frac{b}{b-\gamma}, \frac{-\gamma}{b-\gamma}<1$, then from H$\rm\ddot o$lder's inequality we can get the following interpolation
\begin{align}\label{interpolation11}
     \left\|g\right\|_{L^{2}_{v}}^{2}&\le\varepsilon_{1}\left\|\langle\cdot\rangle^{\frac{b}{2}}g\right\|_{L^{2}_{v}}^{2}+\varepsilon_{1}^{-(1-\frac{b}{b-\gamma})\cdot\frac{b-\gamma}{b}}\left\|\langle\cdot\rangle^{\frac{\gamma}{2}}g\right\|_{L^{2}_{v}}^{2}, \quad \forall \ \varepsilon_{1}>0,
\end{align}
applying \eqref{upper bound1}, the interpolation \eqref{interpolation11} with $g=\omega_{t}A_{+, 1}f$ and $\varepsilon_{1}=\varepsilon^{2}t^{-2(\delta_{j}-1)}t^{2\delta_{1}}$ leads
\begin{align*}
     &\varepsilon^{-1}\delta_{j}^{2} t^{2(\delta_{j}-1)}\left\|\omega_{t} A_{+, 1}f \right\|_{H^{3}_{x}L^{2}_{v}}^{2}\le\varepsilon\delta_{j}^{2}t^{2\delta_{1}}\left\|\langle v\rangle^{\frac{b}{2}}\omega_{t}A_{+, 1}f\right\|_{H^{3}_{x}L^{2}_{v}}^{2}+\frac{\varepsilon^{\frac{2(b-\gamma)}{b}(1-\frac{b}{b-\gamma})-1}\delta_{j}^{2} t^{\theta}}{C_{1}}|||\omega_{t}f|||^{2}.
\end{align*}
here 
$$\theta=\left[\delta_{j}-1-\delta_{1}\left(1-\frac{b}{b-\gamma}\right)\right]\frac{2(b-\gamma)}{b}>\left[\delta_{2}-1-\delta_{1}\left(1-\frac{b}{b-\gamma}\right)\right]\frac{2(b-\gamma)}{b}>0.$$ 
By using \eqref{linear combination-H}, one has
\begin{align*}
     \varepsilon\delta_{j}^{2}t^{2\delta_{1}}\left\|\langle v\rangle^{\frac{b}{2}}\omega_{t}A_{+, 1}f\right\|_{H^{3}_{x}L^{2}_{v}}^{2}&\le2\varepsilon\delta_{j}^{2}\left(\frac{\delta_1+ 1}{\delta_2-\delta_1}\right)^{2}\|\langle v\rangle^{\frac{b}{2}}\omega_{t}H_{\delta_{1}}f\|^{2}_{H^{3}_{x}L^{2}_{v}}\\
     &\quad+2\varepsilon\delta_{j}^{2}t^{2(\delta_{1}-\delta_{2})}\left(\frac{\delta_2+ 1}{\delta_2-\delta_1}\right)^{2}\|\langle v\rangle^{\frac{b}{2}}\omega_{t}H_{\delta_{2}}f\|_{H^{3}_{x}L^{2}_{v}}^{2}.
\end{align*}
The above inequalities yield that for all $0<\varepsilon<1$
\begin{align}\label{111111}
\begin{split}
     &\left|\delta_{j} t^{\delta_{j}-1}(\omega_{t} A_{+, 1}f, \omega_{t} H_{\delta_{j}} f)_{H^3_xL^{2}_v}\right|\le\varepsilon\left\|\omega_{t} H_{\delta_{j}}f\right\|_{H^{3}_{x}L^{2}_{v}}^{2}+2\varepsilon\delta_{j}^{2}\left(\frac{\delta_1+ 1}{\delta_2-\delta_1}\right)^{2}\|\langle v\rangle^{\frac{b}{2}}\omega_{t}H_{\delta_{1}}f\|^{2}_{H^{3}_{x}L^{2}_{v}}\\
     &\qquad+2\varepsilon\delta_{j}^{2}t^{2(\delta_{1}-\delta_{2})}\left(\frac{\delta_2+ 1}{\delta_2-\delta_1}\right)^{2}\|\langle v\rangle^{\frac{b}{2}}\omega_{t}H_{\delta_{2}}f\|_{H^{3}_{x}L^{2}_{v}}^{2}+\frac{\varepsilon^{\frac{2(b-\gamma)}{b}(1-\frac{b}{b-\gamma})-1}\delta_{j}^{2} t^{\theta}}{C_{1}}|||\omega_{t}f|||^{2}.
\end{split}       
\end{align}
Combining these inequalities, it follows that 
\begin{align*}
     &\frac{d}{dt}\|\omega_{t} H_{\delta_{j}}f(t)\|^2_{H^3_xL^{2}_v}+\frac{2c_{0}}{(1+t)^{2}}\|\langle v\rangle^{\frac b2}\omega_{t}H_{\delta_{j}}f(t)\|^{2}_{H^{3}_{x}L^{2}_{v}}+|||\omega_{t} H_{\delta_{j}}f|||^{2}\\
     &\le2C_{6}\left\|\omega_{t} H_{\delta_{j}}f\right\|^{2}_{H^{3}_{x}L^{2}_{v}}+(C_{6})^{2}t^{2\delta_{j}}|||\omega_{t} f|||^{2}+c_{0}b^{2}\delta^{2}_{j} t^{2(\delta_{j}-1)}\|\langle v\rangle^{\frac b2}\omega_{t} f\|^{2}_{H^{3}_{x}L^{2}_{v}}+2C_{4}\|f\|_{H^{3}_{x}L^{2}_{v}}|||\omega_{t} H_{\delta_{j}}f|||^{2}\\
     &\quad+2C_{4}\|H_{\delta_{j}}f\|_{H^{3}_{x}L^{2}_{v}}|||\omega_{t} f|||\cdot|||\omega_{t} H_{\delta_{j}}f|||+2C_{5}\sqrt{C_{0}}t^{\delta_{j}}\|f\|_{H^{3}_{x}L^{2}_{v}}|||\omega_{t} H_{\delta_{j}}f|||+2C_{5}\|H_{\delta_{j}}f\|_{H^{3}_{x}L^{2}_{v}}|||\omega_{t} H_{\delta_{j}}f|||\\
     &\quad+\left\|\omega_{t} H_{\delta_{j}}f\right\|_{H^{3}_{x}L^{2}_{v}}^{2}+\frac{c_{0}}{2(1+T)^{2}}\|\langle v\rangle^{\frac{b}{2}}\omega_{t}H_{\delta_{1}}f\|^{2}_{H^{3}_{x}L^{2}_{v}}+\frac{c_{0}}{2(1+T)^{2}}\|\langle v\rangle^{\frac{b}{2}}\omega_{t}H_{\delta_{2}}f\|_{H^{3}_{x}L^{2}_{v}}^{2}\\
     &\quad+\frac{\varepsilon^{\frac{2(b-\gamma)}{b}(1-\frac{b}{b-\gamma})-1}\delta_{j}^{2} t^{\theta}}{C_{1}}|||\omega_{t}f|||^{2},
\end{align*}
if we choose $\varepsilon$ small enough such that 
$$2\varepsilon\delta_{1}^{2}(T+1)^{2(\delta_{1}-\delta_{2})}\left(\frac{\delta_2+ 1}{\delta_2-\delta_1}\right)^{2}\le\frac{c_{0}}{2(1+T)^{2}}.$$
For all $0 <t\le T$, integrating from 0 to $t$, then by using the Cauchy-Schwarz inequality and \eqref{m=0} yields that for $j=1, 2$
\begin{align*}
     &\|\omega_{t} H_{\delta_{j}}f(t)\|^2_{H^3_xL^{2}_v}+\frac{c_{0}}{(1+T)^{2}}\int_{0}^{t}\|\langle v\rangle^{\frac b2}\omega_{\tau}H_{\delta_{j}}f(\tau)\|^{2}_{H^{3}_{x}L^{2}_{v}}d\tau+\int_{0}^{t}|||\omega_{\tau} H_{\delta_{j}}f(\tau)|||^{2}d\tau\\
     &\le\left(2C_{6}+(4C_{5})^{2}+1\right)\int_{0}^{t}\sup_{j}\left\|\omega_{\tau} H_{\delta_{j}}f(\tau)\right\|^{2}_{H^{3}_{x}L^{2}_{v}}d\tau+\frac12\sup_{j}\|\omega_{t}H_{\delta_{j}}f\|_{L^{\infty}(]0, T]; H^{3}_{x}L^{2}_{v})}^{2}\\
     &\quad+\left(2C_{4}B\epsilon+8(C_{4}B\epsilon)^{2}+\frac14\right)\int_{0}^{t}\sup_{j}|||\omega_{\tau} H_{\delta_{j}}f(\tau)|||^{2}d\tau+C_{7}(B\epsilon)^{2},
\end{align*}
here the constant $C_{7}$ depends on $C_{0}-C_{6}, \ c_{0}, b, \gamma, \delta_{1}, \delta_{2}, \ T$ and we use the fact 
\begin{align}\label{t=0}
      \left\|\omega_{t} H_{\delta_{j}}^{k}f(t)\right\|_{H^3_xL^{2}_v}\Big|_{t=0}=0, \quad \forall \ k\in\mathbb N_{+}, \ j=1, 2.
\end{align}
Taking $\epsilon$ small enough such that 
$$2C_{4}B\epsilon+8(C_{4}B\epsilon)^{2}\le\frac14,$$
we can deduce that for all $t\in]0, T]$ and $j=1, 2$
\begin{align*}
\begin{split}
     &\|\omega_{t} H_{\delta_{j}}f(t)\|^2_{H^3_xL^{2}_v}+\frac{c_{0}}{(1+T)^{2}}\int_{0}^{t}\|\langle v\rangle^{\frac b2}\omega_{\tau}H_{\delta_{j}}f(\tau)\|^{2}_{H^{3}_{x}L^{2}_{v}}d\tau+\frac12\int_{0}^{t}|||\omega_{\tau} H_{\delta_{j}}f(\tau)|||^{2}d\tau\\
     &\le\left(2C_{6}+(4C_{5})^{2}+1\right)\int_{0}^{t}\sup_{j}\left\|\omega_{\tau} H_{\delta_{j}}f(\tau)\right\|^{2}_{H^{3}_{x}L^{2}_{v}}d\tau+\frac12\sup_{j}\|\omega_{t}H_{\delta_{j}}f\|_{L^{\infty}(]0, T]; H^{3}_{x}L^{2}_{v})}^{2}+C_{7}(B\epsilon)^{2},
\end{split}
\end{align*}
this implies that for all $t\in]0, T]$ and $j=1, 2$
\begin{align*}
\begin{split}
     &\|\omega_{t} H_{\delta_{j}}f(t)\|^2_{H^3_xL^{2}_v}+\frac{c_{0}}{(1+T)^{2}}\int_{0}^{t}\|\langle v\rangle^{\frac b2}\omega_{\tau}H_{\delta_{j}}f(\tau)\|^{2}_{H^{3}_{x}L^{2}_{v}}d\tau+\int_{0}^{t}|||\omega_{\tau} H_{\delta_{j}}f(\tau)|||^{2}d\tau\\
     &\le2\left(2C_{6}+(4C_{5})^{2}+1\right)\int_{0}^{t}\sup_{j}\left\|\omega_{\tau} H_{\delta_{j}}f(\tau)\right\|^{2}_{H^{3}_{x}L^{2}_{v}}d\tau+2C_{7}(B\epsilon)^{2}.
\end{split}
\end{align*}
Finally, by using Gronwall inequality, we have for all $0<t\le T$ and $j=1, 2$
\begin{align*}
\begin{split}
     &\|\omega_{t} H_{\delta_{j}}f(t)\|^2_{H^3_xL^{2}_v}+\frac{c_{0}}{(1+T)^{2}}\int_{0}^{t}\|\langle v\rangle^{\frac b2}\omega_{\tau}H_{\delta_{j}}f(\tau)\|^{2}_{H^{3}_{x}L^{2}_{v}}d\tau+\int_{0}^{t}|||\omega_{\tau} H_{\delta_{j}}f(\tau)|||^{2}d\tau\\
     &\le2C_{7}\left(1+2\left(2C_{6}+(4C_{5})^{2}+1\right)e^{2T\left(2C_{6}+(4C_{5})^{2}+1\right)}\right)^{2}\left(B\epsilon\right)^{2}=(\tilde B\epsilon)^{2}.
\end{split}
\end{align*}
\end{proof}
\begin{remark}
Remark that the affirmation of \eqref{t=0} is somehow too simplistic, in fact by using Remark \ref{remark1.2},  the solution belongs to $C^{\infty}([t_0, \infty[; \cap_{s\ge0}H^{\infty, s}_{x, v}(\mathbb T^{3}_{x}\times\mathbb R^{3}_{v}))$ for any $t_0>0$. So we can study the Gevrey smoothness of solution start from initial datum $f(t_0)\in H^{\infty}_{x, v}(\mathbb T^{3}_{x}\times\mathbb R^{3}_{v})$  at $t_0>0$, 
and establish the \`a priori estimate on $[t_0, T]$, but uniformly with respect to parameter $t_0$ (i. e. all constants in the  estimates are independents of small $t_0>0$ ), then in the definition of $H_\delta$, replace $t$ by $t-t_0$, in this case, \eqref{t=0} is true in the following sense, $\forall\ t_0>0$,
$$ 
\lim_{t\, \to \, t_0}\left\|\omega_{t-t_0} H_{\delta_{j}}^{k}f(t)\right\|_{H^3_xL^{2}_v}\le C_k\lim_{t\, \to \, t_0}(t-t_0)^{\delta_j}\left\|\omega_{t-t_0} f(t)\right\|_{H^{k+3}} =0, \quad \forall \ k\ge 1, \ j=1, 2.
$$
\end{remark}

\section{Energy estimates for multi-directional derivations}\label{s5}
This section establishes the energy estimates for multi-directional derivations.
\begin{prop}\label{prop 5.1}
    Assume that $-3<\gamma<0$. Let $f$ be the smooth solution of Cauchy problem \eqref{1-2} with $\|f\|_{L^{\infty}([0, T]; H^3_xL^{2}_v(\omega_{t}))}$ small enough. Then for all $\delta_{1}, \delta_{2}$ satisfy \eqref{delta12}, there exists a constant $A>0$, depends on $\gamma, \ b, \ c_{0}, \ \delta_{1},\ \delta_{2}, \ T$ and $C_{0}-C_{6}$, such that for all $k\ge1$
    \begin{align}\label{4-1-mn}
    \begin{split}
       &\sup_{(m, n)\in E_{k}}\frac{1}{\left((m-2)!(n-2)!\right)^{2\sigma}}\|\omega_{t} H^{m}_{\delta_{1}}H^{n}_{\delta_{2}}f(t)\|^2_{H^3_xL^{2}_v}\\
       &\quad+\frac{c_{0}}{(1+T)^{2}}\sup_{(m, n)\in E_{k}}\frac{1}{\left((m-2)!(n-2)!\right)^{2\sigma}}\int_0^t\|\langle v\rangle^{\frac b2}\omega_{\tau} H^{m}_{\delta_{1}}H^{n}_{\delta_{2}}f(\tau)\|^2_{H^3_xL^{2}_{v}} d\tau\\
       &\quad+\sup_{(m, n)\in E_{k}}\frac{1}{\left((m-2)!(n-2)!\right)^{2\sigma}}\int_0^t|||\omega_{\tau} H^{m}_{\delta_{1}}H^{n}_{\delta_{2}}f(\tau)|||^2 d\tau\le A^{2\sigma(k-\frac 12)}, \quad \forall \ 0<t\le T,
    \end{split}
    \end{align}
    here $E_{k}=\{(m, n)| \ m,n\in\mathbb N, \ 1\le m+n=k\}$.
\end{prop}
\begin{proof}
We prove this proposition by induction on the index $m+n=k$. For the case of $m+n=k=1$, it has already been shown in \eqref{k=1}. By convention, we denote $k!=1$ if $k\le 0$ and $F_{m, n}=\omega_{t} H^{m}_{\delta_{1}}H^{n}_{\delta_{2}}f$. Assume $k\ge2$, for all $1\le m+n=j\le k-1$, 
\begin{align}\label{hypothesis}
\begin{split}
     &\sup_{(m, n)\in E_{j}}\frac{1}{\left((m-2)!(n-2)!\right)^{2\sigma}}\|F_{m, n}(t)\|^2_{H^3_xL^{2}_v}+\sup_{(m, n)\in E_{j}}\frac{1}{\left((m-2)!(n-2)!\right)^{2\sigma}}\int_0^t|||F_{m, n}(\tau)|||^2 d\tau\\
       &\quad+\frac{c_{0}}{(1+T)^{2}}\sup_{(m, n)\in E_{j}}\frac{1}{\left((m-2)!(n-2)!\right)^{2\sigma}}\int_0^t\|\langle v\rangle^{\frac b2}F_{m, n}(\tau)\|^2_{H^3_xL^{2}_{v}} d\tau\le A^{2\sigma(j-\frac 12)}, \quad \forall \ 0<t\le T.
\end{split}
\end{align}
We will show that \eqref{hypothesis} holds for all $m, n\in\mathbb N$ with $m+n=k$.  From \eqref{1-2}, we can obtain
\begin{align*}
     &\frac12\frac{d}{dt}\left\|F_{m, n}(t)\right\|^{2}_{H_{x}^{3}L^{2}_{v}}+\frac{c_{0}}{(1+t)^{2}}\|\langle v\rangle^{\frac b2}F_{m, n}(t)\|^{2}_{H^{3}_{x}L^{2}_{v}}+\left(\omega_{t} H^{m}_{\delta_{1}}H^{n}_{\delta_{2}}\mathcal L_{1}f, \ F_{m, n}\right)_{H^{3}_{x}L^{2}_{v}}\\
     &=-\left(\omega_{t}\left[H^{m}_{\delta_{1}}H^{n}_{\delta_{2}}, \partial_t+v\ \cdot\ \partial_x\right]f, \ F_{m, n}\right)_{H^{3}_{x}L^{2}_{v}}-\left(\omega_{t} H^{m}_{\delta_{1}}H^{n}_{\delta_{2}}\mathcal L_{2}f, \ F_{m, n}\right)_{H^{3}_{x}L^{2}_{v}}\\
     &\quad+\left(\omega_{t} H^{m}_{\delta_{1}}H^{n}_{\delta_{2}}\Gamma(f, f), \ F_{m, n}\right)_{H^{3}_{x}L^{2}_{v}}.
\end{align*}
If $m=0, n=k$ or $n=0, m=k$, then the commutator in the above formula has been given in \eqref{kehigher}. For simplicity of the presentation, we consider the case of $m, n\ge1$ with $m+n=k$. The proof is similar and relatively easy in the case of $m=0, n=k$ and $n=0, m=k$.

Applying \eqref{commutator-mn}, we can obtain that
\begin{align*}
     &\left|\left(\omega_{t} \left[H^{m}_{\delta_{1}}H^{n}_{\delta_{2}}, \partial_t+v\ \cdot\ \partial_x\right]f, \ F_{m, n}\right)_{H^{3}_{x}L^{2}_{v}}\right|\\
     &\le\delta_{1}mt^{\delta_{1}-1}\left|\left(A_{+, 1}F_{m-1, n}, \ F_{m, n}\right)_{H^{3}_{x}L^{2}_{v}}\right|+\delta_{2}nt^{\delta_{2}-1}\left|\left(A_{+, 1}F_{m, n-1}, \ F_{m, n}\right)_{H^{3}_{x}L^{2}_{v}}\right|=\mathcal J_{1}(t)+\mathcal J_{2}(t).
\end{align*}
Since $m, n\ge1$, from Proposition \ref{H-m-Gamma}, we can obtain that
\begin{align*}
     &\left|\left(\omega_{t} H^{m}_{\delta_{1}}H^{n}_{\delta_{2}}\Gamma(f, f),\ F_{m, n}\right)_{H^{3}_{x}L^{2}_{v}}\right|\le C_{4}\|f(t)\|_{H^{3}_{x}L^{2}_{v}}|||F_{m, n}(t)|||^{2}\\
     &\qquad+C_{4}\|F_{m, n}(t)\|_{H^{3}_{x}L^{2}_{v}}|||\omega_{t} f(t)|||\cdot|||F_{m, n}(t)|||+\mathcal R_{1}(t),
\end{align*}
with
\begin{align}\label{R-1}
\begin{split}
     \mathcal R_{1}(t)&=C_{4}\sum_{l=1}^{m}\sum_{p=0}^{n-1}\binom{m}{l}\binom{n}{p}\|F_{l, p}(t)\|_{H^{3}_{x}L^{2}_{v}}|||F_{m-l, n-p}(t)|||\cdot|||F_{m, n}(t)|||\\
     &\quad+C_{4}\sum_{l=1}^{m-1}\binom{m}{l}\|H^{l}_{\delta_{1}}H^{n}_{\delta_{2}}f(t)\|_{H^{3}_{x}L^{2}_{v}}|||F_{m-l, 0}(t)|||\cdot|||F_{m, n}(t)|||\\
     &\quad+C_{4}\sum_{p=1}^{n}\binom{n}{p}\|H^{p}_{\delta_{2}}f(t)\|_{H^{3}_{x}L^{2}_{v}}|||F_{m, n-p}(t)|||\cdot|||F_{m, n}(t)|||.
\end{split}
\end{align}
Since $m, n\ge1$, from Corollary \ref{L2}, by using Cauchy-Schwarz inequality, we have
\begin{align*}
     &|(\omega_{t}H^{m}_{\delta_{1}}H^{n}_{\delta_{2}}\mathcal L_{2}f,\ F_{m, n})_{H^{3}_{x}L^{2}_{v}}|
     \le(C_{5})^{2}\|F_{m, n}(t)\|^{2}_{H^{3}_{x}L^{2}_{v}}+\frac{1}{4}|||F_{m, n}(t)|||^{2}+\mathcal R_{2}(t),
\end{align*}
with
\begin{align}\label{R-2}
\begin{split}
     &\mathcal R_{2}(t)=C_{5}\sum_{p=0}^{n-1}\binom{n}{p}\left(\sqrt{C_{0}}t^{\delta_{2}}\right)^{n-p}\sqrt{(n-p)!}\|F_{m, p}(t)\|_{H^{3}_{x}L^{2}_{v}}|||F_{m, n}(t)|||\\
     &\ +C_{5}\sum_{l=0}^{m-1}\sum_{p=0}^{n}\binom{m}{l}\binom{n}{p}\left(\sqrt{C_{0}}t^{\delta_{1}}\right)^{m-l}\left(\sqrt{C_{0}}t^{\delta_{2}}\right)^{n-p}\sqrt{(m-l+n-p+3)!}\|F_{l, p}(t)\|_{H^{3}_{x}L^{2}_{v}}|||F_{m, n}(t)|||.
\end{split}
\end{align}
Since $\gamma<0$, from Lemma \ref{L1} and Proposition \ref{H-m-L-1}, we can get that
\begin{align*}
     \left(\omega_{t} H^{m}_{\delta_{1}}H^{n}_{\delta_{2}}\mathcal L_{1}f,\ F_{m, n}\right)_{H^{3}_{x}L^{2}_{v}}
     &\ge \left(\mathcal L_{1}F_{m, n},\ F_{m, n}\right)_{H^{3}_{x}L^{2}_{v}}-\left|\left([\omega_{t} H^{m}_{\delta_{1}}H^{n}_{\delta_{2}}, \mathcal L_{1}]f,\ F_{m, n}\right)_{H^{3}_{x}L^{2}_{v}}\right|\\
     &\ge\frac{3}{4}|||F_{m, n}(t)|||^{2}-\tilde C_{6}\|F_{m, n}(t)\|^{2}_{H^{3}_{x}L^{2}_{v}}-\mathcal R_{3}(t),
\end{align*}
with
\begin{align}\label{R-3}
\begin{split}
     \mathcal R_{3}(t)&=C_{6}\sum_{l=1}^{m}\binom{m}{l}t^{\delta_{1}l}\sqrt{(l+1)!}|||F_{m-l, n}(t)|||\cdot|||F_{m, n}(t)|||\\
     &\quad+C_{6}\sum_{p=1}^{n}\binom{n}{p}t^{\delta_{2}p}\sqrt{(p+1)!}|||F_{m, n-p}(t)|||\cdot|||F_{m, n}(t)|||\\
     &\quad+C_{6}\sum_{l=1}^{m}\sum_{p=1}^{n}\binom{m}{l}\binom{n}{p}t^{\delta_{1}l}t^{\delta_{2}p}\sqrt{(l+p+1)!}|||F_{m-l, n-p}(t)|||\cdot|||F_{m, n}(t)|||.
\end{split}
\end{align}
Combining the above results, it follows that
\begin{align*}
     &\frac{d}{dt}\left\|F_{m, n}f(t)\right\|^{2}_{H_{x}^{3}L^{2}_{v}}+\frac{2c_{0}}{(1+t)^{2}}\|\langle v\rangle^{\frac b2}F_{m, n}(t)\|^{2}_{H^{3}_{x}L^{2}_{v}}+|||F_{m, n}(t)|||^{2}\\
     &\le2(\tilde C_{6}+(C_{5})^{2})\|F_{m, n}(t)\|^{2}_{H^{3}_{x}L^{2}_{v}}+2C_{4}\|f(t)\|_{H^{3}_{x}L^{2}_{v}}|||F_{m, n}(t)|||^{2}+2\mathcal J_{1}(t)+2\mathcal J_{2}(t)\\
     &\quad+2C_{4}\|F_{m, n}(t)\|_{H^{3}_{x}L^{2}_{v}}|||\omega_{t} f(t)|||\cdot|||F_{m, n}(t)|||+2\mathcal R_{1}(t)+2\mathcal R_{2}(t)+2\mathcal R_{3}(t).
\end{align*}
For all $0<t\le T$, integrating from~0~ to $t$, since $\|f\|_{L^{\infty}([0, T]; H^2_xL^{2}_v(\omega_{t}))}$ small enough, then by using \eqref{t=0}, one has for all $0<\epsilon<1$
\begin{align*}
    &\left\|F_{m, n}(t)\right\|^{2}_{H_{x}^{3}L^{2}_{v}}+\frac{2c_{0}}{(1+T)^{2}}\int_{0}^{t}\|\langle v\rangle^{\frac b2}F_{m, n}(\tau)\|^{2}_{H^{3}_{x}L^{2}_{v}}d\tau+\int_{0}^{t}|||F_{m, n}(\tau)|||^{2}d\tau\\
     &\le2(\tilde C_{6}+(C_{5})^{2})\int_{0}^{t}\|F_{m, n}(\tau)\|^{2}_{H^{3}_{x}L^{2}_{v}}d\tau+4C_{4}\epsilon\int_{0}^{t}|||F_{m, n}(\tau)|||^{2}d\tau+2\int_{0}^{t}\mathcal J_{1}(\tau)d\tau\\
     &\quad+2\int_{0}^{t}\mathcal J_{2}(\tau)d\tau+2\int_{0}^{t}\mathcal R_{1}(\tau)d\tau+2\int_{0}^{t}\mathcal R_{2}(\tau)d\tau+2\int_{0}^{t}\mathcal R_{3}(\tau)d\tau+C_{4}\epsilon\|F_{m, n}\|^{2}_{L^{\infty}(]0, T]; H^{3}_{x}L^{2}_{v})},
\end{align*}
Taking $4C_{4}\epsilon\le\frac12$, then we can deduce that for all $0<t\le T$
\begin{align*}
    &\left\|F_{m, n}(t)\right\|^{2}_{H_{x}^{3}L^{2}_{v}}+\frac{4c_{0}}{(1+T)^{2}}\int_{0}^{t}\|\langle v\rangle^{\frac b2}F_{m, n}(\tau)\|^{2}_{H^{3}_{x}L^{2}_{v}}d\tau+\int_{0}^{t}|||F_{m, n}(\tau)|||^{2}d\tau\\
     &\le4(\tilde C_{6}+(C_{5})^{2})\int_{0}^{t}\|F_{m, n}(\tau)\|^{2}_{H^{3}_{x}L^{2}_{v}}d\tau+4\int_{0}^{t}\mathcal J_{1}(\tau)d\tau+4\int_{0}^{t}\mathcal J_{2}(\tau)d\tau\\
     &\quad+4\int_{0}^{t}\mathcal R_{1}(\tau)d\tau+4\int_{0}^{t}\mathcal R_{2}(\tau)d\tau+4\int_{0}^{t}\mathcal R_{3}(\tau)d\tau,
\end{align*}
this implies that for all $(m, n)\in E_{k}$
\begin{align}\label{integration-k}
\begin{split}
    &\frac{1}{\left((m-2)!(n-2)!\right)^{2\sigma}}\left\|F_{m, n}(t)\right\|^{2}_{H_{x}^{3}L^{2}_{v}}+\frac{1}{\left((m-2)!(n-2)!\right)^{2\sigma}}\int_{0}^{t}|||F_{m, n}(\tau)|||^{2}d\tau\\
    &\quad+\frac{4c_{0}}{(1+T)^{2}}\frac{1}{\left((m-2)!(n-2)!\right)^{2\sigma}}\int_{0}^{t}\|\langle v\rangle^{\frac b2}F_{m, n}(\tau)\|^{2}_{H^{3}_{x}L^{2}_{v}}d\tau\\
     &\le\sup_{(m, n)\in E_{k}}\frac{4(\tilde C_{6}+(C_{5})^{2})}{\left((m-2)!(n-2)!\right)^{2\sigma}}\int_{0}^{t}\|F_{m, n}(\tau)\|^{2}_{H^{3}_{x}L^{2}_{v}}d\tau+J_{1}+J_{2}+R_{1}+R_{2}+R_{3},
\end{split}
\end{align}
with
\begin{align*}
     J_{s}=\sup_{(m, n)\in E_{k}}\frac{4}{\left((m-2)!(n-2)!\right)^{2\sigma}}\int_{0}^{t}\mathcal J_{s}(\tau)d\tau, \quad s=1, 2,
\end{align*}
and 
\begin{align*}
     R_{s}=\sup_{(m, n)\in E_{k}}\frac{4}{\left((m-2)!(n-2)!\right)^{2\sigma}}\int_{0}^{t}\mathcal R_{s}(\tau)d\tau, \quad s=1, 2, 3.
\end{align*}
We estimate the terms of the right-hand side of \eqref{integration-k} by the following lemmas.
\begin{lemma}\label{J1-J2}
    Assume that $f$ satisfies \eqref{hypothesis}, then for all $0<t\le T$
\begin{align}\label{J1-1}
\begin{split}
     J_{1}
     &\le\frac{c_{0}}{(1+T)^{2}}\sup_{(m, n)\in E_{k}}\frac{1}{\left((m-2)!(n-2)!\right)^{2\sigma}}\int_{0}^{t}\left\|\langle v\rangle^{\frac{b}{2}}F_{m, n}(\tau)\right\|_{H^{3}_{x}L^{2}_{v}}^{2}d\tau\\
     &\quad+\sup_{(m, n)\in E_{k}}\frac{C_{1}}{\left(64(m-2)!(n-2)!\right)^{2\sigma}}\int_{0}^{t}\left\|F_{m, n}(\tau)\right\|_{H^{3}_{x}L^{2}_{v}}^{2}d\tau+\left(A_{1}A^{k-\frac32}\right)^{2\sigma}.
\end{split}
\end{align}
\begin{align}\label{J2-1}
\begin{split}
     J_{2}
     &\le\frac{c_{0}}{(1+T)^{2}}\sup_{(m, n)\in E_{k}}\frac{1}{\left((m-2)!(n-2)!\right)^{2\sigma}}\int_{0}^{t}\left\|\langle v\rangle^{\frac{b}{2}}F_{m, n}(\tau)\right\|_{H^{3}_{x}L^{2}_{v}}^{2}d\tau\\
     &\quad+\sup_{(m, n)\in E_{k}}\frac{C_{1}}{\left(64(m-2)!(n-2)!\right)^{2\sigma}}\int_{0}^{t}\left\|F_{m, n}(\tau)\right\|_{H^{3}_{x}L^{2}_{v}}^{2}d\tau+\left(A_{2}A^{k-\frac32}\right)^{2\sigma}.
\end{split}
\end{align}
with $A_{1}, A_{2}$ depends on $c_{0}, \ b, \ \delta_{1}, \ \delta_{2}, \ C_{0}-C_{6}$ and $T$.
\end{lemma}
and
\begin{lemma}\label{R1-3}
    Assume that $f$ satisfies \eqref{m=0} and \eqref{hypothesis}, then for all $0<t\le T$
\begin{align}\label{R1}
     R_{1}\le\sup_{(m, n)\in E_{k}}\frac{1}{16\left((m-2)!(n-2)!\right)^{2\sigma}}\int_{0}^{t}|||F_{m, n}(\tau)|||^{2}d\tau+\left(A_{3}A^{k-1}\right)^{2\sigma},
\end{align}
with the constant $A_{3}$ depends on $\gamma, b$ and $C_{4}$. And
\begin{align}\label{R2}
     R_{2}\le\sup_{(m, n)\in E_{k}}\frac{1}{16\left((m-2)!(n-2)!\right)^{2\sigma}}\int_{0}^{t}|||F_{m, n}(\tau)|||^{2}d\tau+\left(A_{4}A^{k-1}\right)^{2\sigma},
\end{align}
with the constant $A_{4}$ depends on $\gamma, \ b, T, C_{0}$ and $C_{5}$. And
\begin{align}\label{R3}
     R_{3}\le\sup_{(m, n)\in E_{k}}\frac{1}{16\left((m-2)!(n-2)!\right)^{2\sigma}}\int_{0}^{t}|||F_{m, n}(\tau)|||^{2}d\tau+\left(A_{5}A^{k-1}\right)^{2\sigma},
\end{align}
with the constant $A_{5}$ depends on $\gamma, \ b$ and $C_{6}$.
\end{lemma}
{\bf End of Proof of Proposition \ref{prop 5.1}. } Plugging \eqref{J1-1}, \eqref{J2-1}, \eqref{R1}, \eqref{R2} and \eqref{R3} back into \eqref{integration-k}, it follows that for all $(m, n)\in E_{k}$
\begin{align}\label{integration-fin}
\begin{split}
    &\frac{1}{\left((m-2)!(n-2)!\right)^{2\sigma}}\left\|F_{m, n}(t)\right\|^{2}_{H_{x}^{3}L^{2}_{v}}+\frac{1}{\left((m-2)!(n-2)!\right)^{2\sigma}}\int_{0}^{t}|||F_{m, n}(\tau)|||^{2}d\tau\\
    &\quad+\frac{c_{0}}{(1+T)^{2}}\frac{1}{\left((m-2)!(n-2)!\right)^{2\sigma}}\int_{0}^{t}\|\langle v\rangle^{\frac b2}F_{m, n}(\tau)\|^{2}_{H^{3}_{x}L^{2}_{v}}d\tau\\
     &\le\sup_{(m, n)\in E_{k}}\frac{8(\tilde C_{6}+(C_{5})^{2})+C_{1}}{\left((m-2)!(n-2)!\right)^{2\sigma}}\int_{0}^{t}\|F_{m, n}(\tau)\|^{2}_{H^{3}_{x}L^{2}_{v}}d\tau+2\left(A_{0}A^{k-1}\right)^{2\sigma},
\end{split}
\end{align}
if we choose $A\ge1$, here $A_{0}=A_{1}+A_{2}+A_{3}+A_{4}+A_{5}$.
Using Gronwall inequality, one has
\begin{align*}
    &\sup_{(m, n)\in E_{k}}\frac{8(\tilde C_{6}+(C_{5})^{2})+C_{1}}{\left((m-2)!(n-2)!\right)^{2\sigma}}\int_{0}^{t}\|F_{m, n}(\tau)\|^{2}_{H^{3}_{x}L^{2}_{v}}d\tau\\
    &\le2\left(1+\left(8(\tilde C_{6}+(C_{5})^{2})+C_{1}\right)e^{8(\tilde C_{6}+(C_{5})^{2})+C_{1}T}\right)\left(A_{0}A^{k-1}\right)^{2\sigma},
\end{align*}
plugging it back into \eqref{integration-fin}, one can deduce
\begin{align*}
    &\frac{1}{\left((m-2)!(n-2)!\right)^{2\sigma}}\left\|F_{m, n}(t)\right\|^{2}_{H_{x}^{3}L^{2}_{v}}+\frac{1}{\left((m-2)!(n-2)!\right)^{2\sigma}}\int_{0}^{t}|||F_{m, n}(\tau)|||^{2}d\tau\\
    &\quad+\frac{c_{0}}{(1+T)^{2}}\frac{1}{\left((m-2)!(n-2)!\right)^{2\sigma}}\int_{0}^{t}\|\langle v\rangle^{\frac b2}F_{m, n}(\tau)\|^{2}_{H^{3}_{x}L^{2}_{v}}d\tau\\
     &\le2\left(1+\left(8(\tilde C_{6}+(C_{5})^{2})+C_{1}\right)e^{8(\tilde C_{6}+(C_{5})^{2})+C_{1}T}\right)^{2}\left(A_{0}A^{k-1}\right)^{2\sigma}.
\end{align*}
We prove then
\begin{align*}
    &\sup_{(m, n)\in E_{k}}\frac{\|F_{m, n}(t)\|^2_{H^3_xL^{2}_v}}{\left((m-2)!(n-2)!\right)^{2\sigma}}+\frac{c_{0}}{(1+T)^{2}}\sup_{(m, n)\in E_{k}}\frac{1}{\left((m-2)!(n-2)!\right)^{2\sigma}}\int_0^t\|\langle v\rangle^{\frac b2}F_{m, n}(\tau)\|^2_{H^3_xL^{2}_{v}} d\tau\\
       &\quad+\sup_{(m, n)\in E_{k}}\frac{1}{\left((m-2)!(n-2)!\right)^{2\sigma}}\int_0^t|||F_{m, n}(\tau)|||^2 d\tau\le A^{2\sigma(k-\frac 12)}, \quad \forall \ 0<t\le T,
\end{align*}
if we choose the constant $A$ such that
\begin{align*}
     A\ge2\left(1+\left(8(\tilde C_{6}+(C_{5})^{2})+C_{1}\right)e^{8(\tilde C_{6}+(C_{5})^{2})+C_{1}T}\right)A_{0}.
\end{align*}
\end{proof}

\bigskip

{\bf Proof of \eqref{Gelfand-Shilov}:} 
Setting $\lambda>2\max\{1, \frac{b-\gamma}{2b}\}$, define $\delta_{1}, \ \delta_{2}$ satisfies \eqref{delta12}. Then $\delta_2>\delta_1>\lambda$. With $\delta_1$ and  $\delta_2$ given in \eqref{delta12}, we have
\begin{equation*}
	H_{\delta_1}=\frac{1}{\delta_1+1} t^{\delta_1+1}\partial_{x_1}-t^{\delta_1}A_{+, 1}   ,\quad
  H_{\delta_2}=\frac{1}{\delta_2+ 1} t^{\delta_2+1}\partial_{x_1}-t^{\delta_2}A_{+, 1}.
\end{equation*}
Let $f$ be the smooth solution of the Cauchy problem \eqref{1-2} satisfying $\|f_{0}\|_{H^{3}_{x}L^{2}_{v}(\omega_{0})}$ small, from \eqref{linear combination-H}, then for all $\alpha_{1}, m\in\mathbb N$ and $0<t\le T$
\begin{equation}\label{A}
\begin{split}
     &t^{(\lambda+1)\alpha_{1}+\lambda m}\|\omega_{t}\partial^{\alpha_{1}}_{x_{1}}A^{m}_{+, 1}f(t)\|_{H^{3}_{x}L^{2}_{v}}=\|\omega_{t}(T_{1}+T_{2})^{\alpha_{1}}(T_{3}+T_{4})^{m}f(t)\|_{H^{3}_{x}L^{2}_{v}}\\
     &\le\sum_{j=0}^{\alpha_{1}}\sum_{k=0}^{m}\binom{\alpha_{1}}{j}\binom{m}{k}\left|\frac{(\delta_2+ 1)(\delta_1+1)}{\delta_2-\delta_1}\right|^{\alpha_{1}+m}(T+1)^{(\delta_{1}-\delta_{2})(\alpha_{1}-j+m-k)}\|\omega_{t} H^{j+k}_{\delta_{1}}H^{\alpha_{1}+m-j-k}_{\delta_{2}}f(t)\|_{H^{3}_{x}L^{2}_{v}}.
\end{split}
\end{equation}
From Proposition \ref{prop 5.1}, From Proposition \ref{prop 5.1}, we have that for all $\alpha_{1}, m\in\mathbb Z_{+}$
\begin{align*}
       &\sup_{(p, q)\in E_{m+\alpha_{1}}}\frac{1}{\left((p-2)!(q-2)!\right)^{2\sigma}}\|\omega_{t} H^{p}_{\delta_{1}}H^{q}_{\delta_{2}}f(t)\|^2_{H^3_xL^{2}_v}\le A^{2\sigma(m+\alpha_{1}-\frac 12)}, \quad \forall \ 0<t\le T,
\end{align*}
this yields that for all $\alpha_{1}, m\in\mathbb Z_{+}$
\begin{align*}
       &\frac{1}{\left((j+k-2)!(\alpha_{1}+m-j-k-2)!\right)^{2\sigma}}\|\omega_{t} H^{j+k}_{\delta_{1}}H^{\alpha_{1}+m-j-k}_{\delta_{2}}f(t)\|_{H^{3}_{x}L^{2}_{v}}^{2}\le A^{2\sigma(m+\alpha_{1}-\frac 12)}, \quad \forall \ 0<t\le T,
\end{align*}
thus, we have that for all $0<t\le T$ and $\alpha_{1}, m\in\mathbb Z_{+}$
\begin{equation*}
\begin{split}
     \|\omega_{t} H^{j+k}_{\delta_{1}}H^{\alpha_{1}+m-j-k}_{\delta_{2}}f(t)\|_{H^{3}_{x}L^{2}_{v}}&\le \left(A^{\alpha_1+m-\frac12}(j+k-2)!(\alpha_{1}+m-j-k-2)!\right)^{\sigma}\\
     &\le \left(A^{\alpha_{1}+m-\frac12}(\alpha_{1}+m)!\right)^{\sigma},
\end{split}
\end{equation*}
with $j=0, 1, \cdots, \alpha_{1}$ and $k=0, 1, \cdots, m$, here we use the fact that $p!q!\le(p+q)!$. Plugging it back into \eqref{A}, since $\delta_{1}>\delta_{2}$ and $A\ge1$, then one can deduce that for all $0<t\le T$
\begin{equation*}
\begin{split}
     &t^{(\lambda+1)\alpha_{1}+\lambda m}\|\omega_{t}\partial^{\alpha_{1}}_{x_{1}}A^{m}_{+, 1}f(t)\|_{H^{3}_{x}L^{2}_{v}}\\
     &\le\sum_{j=0}^{\alpha_{1}}\sum_{k=0}^{m}\binom{\alpha_{1}}{j}\binom{m}{k}\left(\frac{(\delta_2+ 1)(\delta_1+1)}{\delta_2-\delta_1}\right)^{\alpha_{1}+m}t^{(\delta_{1}-\delta_{2})(\alpha_{1}+m)}\left(A^{\alpha_{1}+m-\frac12}(\alpha_{1}+m)!\right)^{\sigma}\\
     &\le\left(2A^{\sigma}(T+1)^{\delta_{1}-\delta_{2}}\frac{(\delta_2+ 1)(\delta_1+1)}{\delta_2-\delta_1}\right)^{\alpha_{1}+m} \left((\alpha_{1}+m)!\right)^{\sigma}.
\end{split}
\end{equation*}
Similarly, the above inequality is also true for $\partial^{\alpha_{1}}_{x_{j}}A^{m}_{+, j}$ with $j=2, 3$, and obtain
\begin{equation*}
\begin{split}
     &\|\omega_{t}{\partial_x^{\alpha} \nabla_{\mathcal H_{+}}^{m}  f(t)}\|^{2}_{H^3_{x}L^2_v}=\sum_{|\beta|=m}\frac{m!}{\beta!}\|\omega_{t}{\partial_x^{\alpha} A_+^{\beta}  f(t)}\|^2_{H^3_{x}L^2_v}\\
     &\le\sum_{|\beta|=m}\frac{m!}{\beta!}\left(\sum_{j=1}^3\|\omega_{t}{\partial^{|\alpha|}_{x_{j}}A^{m}_{+, j}f(t)}\|_{H^3_{x}L^2_v}\right)^2\le 3^{m}\left(\sum_{j=1}^3\|\omega_{t}{\partial^{|\alpha|}_{x_{j}}A^{m}_{+, j}f(t)}\|_{H^3_{x}L^2_v}\right)^2,
\end{split}
\end{equation*}
here we use
$$\sum_{|\beta|=m}\frac{m!}{\beta!}=3^{m}, \quad \beta\in\mathbb N^{3}.$$
And therefore,  for any $0<t\le T$
\begin{align*}
\begin{split}
     &t^{(\lambda+1)|\alpha|+ \lambda m} \|\omega_{t}{\partial_x^{\alpha} \nabla_{\mathcal H_{+}}^{m}  f(t)}\|_{H^3_{x}L^2_v}
     \le t^{(\lambda+1)|\alpha|+ \lambda m}3^{\frac{m}{2}}\sum_{j=1}^{3}\|\omega_{t}\partial^{|\alpha|}_{x_{j}}A^{m}_{+, j}f(t)\|_{H^{3}_{x}L^{2}_{v}}\\
     &\le3\left(6A^{\sigma}2^{\delta_1-\delta_2}\frac{(\delta_2+ 1)(\delta_1+1)}{\delta_1-\delta_2} \right)^{|\alpha|+m}\left((|\alpha|+m)!\right)^{\sigma}\leq C^{|\alpha|+m+1} \left( (|\alpha|+m)!\right)^{\sigma},
\end{split}
\end{align*}
here $C\ge\max\{3, 6A^{\sigma}2^{\delta_1-\delta_2}\frac{(\delta_2+ 1)(\delta_1+1)}{\delta_1-\delta_2} \}$. 

\section{Proofs of technical lemmas}

In this section, we prove Lemma \ref{J1-J2} and Lemma \ref{R1-3}.
\bigskip

{\bf Proof of Lemma \ref{J1-J2}.} Applying \eqref{interpolation11} with $g=\omega_{t}A_{+, 1}H^{m-1}_{\delta_{1}}H^{n}_{\delta_{2}}f$ and $\varepsilon_{1}=\varepsilon^{2}t^{-2(\delta_{j}-1)}t^{2\delta_{1}}m^{-2}$, similar to \eqref{111111}, we can obtain that 
\begin{align*}
     &\mathcal J_{1}(t)
     \le\varepsilon\left\|F_{m, n}\right\|_{H^{3}_{x}L^{2}_{v}}^{2}+2\varepsilon\delta_{j}^{2}\left(\frac{\delta_1+ 1}{\delta_2-\delta_1}\right)^{2}\|\langle v\rangle^{\frac{b}{2}}F_{m, n}\|^{2}_{H^{3}_{x}L^{2}_{v}}\\
     &\quad+2\varepsilon\delta_{j}^{2}t^{2(\delta_{1}-\delta_{2})}\left(\frac{\delta_2+ 1}{\delta_2-\delta_1}\right)^{2}\|\langle v\rangle^{\frac{b}{2}}F_{m-1, n+1}\|_{H^{3}_{x}L^{2}_{v}}^{2}+\frac{\varepsilon^{\frac{2b}{b-\gamma}(1-\frac{b}{b-\gamma})-1}\delta_{j}^{2}m^{\frac{b-\gamma}{b}} t^{\theta}}{C_{1}}|||F_{m-1, n}|||^{2}.
\end{align*}
taking $0<\varepsilon<1$ small enough such that 
$$\varepsilon=\min\left\{\frac{1}{4C_{\gamma, b, \delta_{1}, \delta_{2}, T}}, \frac{C_{1}}{64C_{\gamma, b, \delta_{1}, \delta_{2}, T}}\left(\frac{c_{0}}{(1+T)^{2}}\right)^{-\frac\gamma b}\right\}.$$
Thus, by using the hypothesis \eqref{hypothesis}, we can get 
\begin{align*}
     J_{1}
     &\le\frac{c_{0}}{(1+T)^{2}}\sup_{(m, n)\in E_{k}}\frac{1}{\left((m-2)!(n-2)!\right)^{2\sigma}}\int_{0}^{t}\left\|\langle v\rangle^{\frac{b}{2}}F_{m, n}(\tau)\right\|_{H^{3}_{x}L^{2}_{v}}^{2}d\tau\\
     &\quad+\sup_{(m, n)\in E_{k}}\frac{C_{1}}{\left(64(m-2)!(n-2)!\right)^{2\sigma}}\int_{0}^{t}\left\|F_{m, n}(\tau)\right\|_{H^{3}_{x}L^{2}_{v}}^{2}d\tau\\
     &\quad+C_{8}\sup_{(m, n)\in E_{k}}\frac{m^{2\sigma}}{\left((m-2)!(n-2)!\right)^{2\sigma}}\int_{0}^{t}\left\|F_{m-1, n}(\tau)\right\|_{H^{3}_{x}L^{2}_{v}}^{2}d\tau\\
     &\le\frac{c_{0}}{(1+T)^{2}}\sup_{(m, n)\in E_{k}}\frac{1}{\left((m-2)!(n-2)!\right)^{2\sigma}}\int_{0}^{t}\left\|\langle v\rangle^{\frac{b}{2}}F_{m, n}(\tau)\right\|_{H^{3}_{x}L^{2}_{v}}^{2}d\tau\\
     &\quad+\sup_{(m, n)\in E_{k}}\frac{C_{1}}{\left(64(m-2)!(n-2)!\right)^{2\sigma}}\int_{0}^{t}\left\|F_{m, n}(\tau)\right\|_{H^{3}_{x}L^{2}_{v}}^{2}d\tau+\left(A_{1}A^{k-\frac32}\right)^{2\sigma}.
\end{align*}
Similarly, one can deduce that for all $(m, n)\in E_{k}$
\begin{align*}
     J_{2}
     &\le\frac{c_{0}}{(1+T)^{2}}\sup_{(m, n)\in E_{k}}\frac{1}{\left((m-2)!(n-2)!\right)^{2\sigma}}\int_{0}^{t}\left\|\langle v\rangle^{\frac{b}{2}}F_{m, n}(\tau)\right\|_{H^{3}_{x}L^{2}_{v}}^{2}d\tau\\
     &\quad+\sup_{(m, n)\in E_{k}}\frac{C_{1}}{\left(64(m-2)!(n-2)!\right)^{2\sigma}}\int_{0}^{t}\left\|F_{m, n}(\tau)\right\|_{H^{3}_{x}L^{2}_{v}}^{2}d\tau+\left(A_{2}A^{k-\frac32}\right)^{2\sigma}.
\end{align*}

\bigskip

{\bf Proof of \eqref{R1}.} From the Cauchy-Schwarz inequality, one has
\begin{align*}
     \int_{0}^{t}\mathcal R_{1}(\tau)d\tau
     &\le48\left(C_{4}\sum_{l=1}^{m}\sum_{p=0}^{n-1}\binom{m}{l}\binom{n}{p}\|F_{l, p}\|_{L^{\infty}(]0, T]; H^{3}_{x}L^{2}_{v})}\left(\int_{0}^{t}|||F_{m-l, n-p}(\tau)|||^{2}\right)^{\frac12}\right)^{2}\\
     &\quad+48\left(C_{4}\sum_{l=1}^{m-1}\binom{m}{l}\|F_{l, n}\|_{L^{\infty}(]0, T]; H^{3}_{x}L^{2}_{v})}\left(\int_{0}^{t}|||F_{m-l, 0}(\tau)|||^{2}d\tau\right)^{\frac12}\right)^{2}\\
     &\quad+48\left(C_{4}\sum_{p=1}^{n}\binom{n}{p}\|F_{0, p}\|_{L^{\infty}(]0, T]; H^{3}_{x}L^{2}_{v})}\left(\int_{0}^{t}|||F_{m, n-p}(\tau)|||^{2}d\tau\right)^{\frac12}\right)^{2}\\
     &\quad+\frac{1}{64}\int_{0}^{t}|||F_{m, n}(\tau)|||^{2}d\tau=48(R_{1, 1})^{2}+48(R_{1, 2})^{2}+48(R_{1, 3})^{2}+\frac{1}{64}\int_{0}^{t}|||F_{m, n}(\tau)|||^{2}d\tau.
\end{align*}
It follows from the hypothesis \eqref{hypothesis} that for all $(m, n)\in E_{k}$
\begin{align*}
     R_{1, 1}&=C_{4}\sum_{l=1}^{m}\sum_{p=0}^{n-1}\frac{m!\left((l-2)!(m-l-2)!\right)^{\sigma}}{l!(m-l)!}\frac{n!\left((p-2)!(n-p-2)!\right)^{\sigma}}{p!(n-p)!}\\
     &\quad\times\frac{\|F_{l, p}\|_{L^{\infty}(]0, T]; H^{3}_{x}L^{2}_{v})}}{\left((l-2)!(m-l-2)!\right)^{\sigma}}\frac{\left(\int_{0}^{t}|||F_{m-l, n-p}(\tau)|||^{2}d\tau \right)^{\frac12}}{\left((p-2)!(n-p-2)!\right)^{\sigma}}\\
     &\le C_{4}\sum_{l=1}^{m}\sum_{p=0}^{n-1}\frac{m!\left((l-2)!(m-l-2)!\right)^{\sigma}}{l!(m-l)!}\frac{n!\left((p-2)!(n-p-2)!\right)^{\sigma}}{p!(n-p)!}A^{\sigma(k-1)}.
\end{align*}
Since $p!q!\le(p+q)!$ for all $p, q\in\mathbb N$, then
\begin{align}\label{sum1}
\begin{split}
     \sum_{l=2}^{m-2}\frac{m!\left((l-2)!(m-l-2)!\right)^{\sigma}}{l!(m-l)!}&=(m-2)!\sum_{l=2}^{m-2}\frac{m(m-1)\left((l-2)!(m-l-2)!\right)^{\sigma-1}}{l(l-1)(m-l)(m-l-1)}\\
     &\le\left((m-2)!\right)^{\sigma}\sum_{l=2}^{m-2}\frac{m(m-1)}{l(l-1)(m-l)(m-l-1)}\\
     &\le 16\left((m-2)!\right)^{\sigma}.
\end{split}     
\end{align}
Hence, for all $(m, n)\in E_{k}$
\begin{align*}
     R_{1, 1}\le C_{4}\left(25^{2}A^{k-1}(m-2)!(n-2)!\right)^{\sigma}.
\end{align*}
Similarly, one can deduce that for all $(m, n)\in E_{k}$
\begin{align*}
     R_{1, 2}\le C_{4}\left(25A^{k-1}(m-2)!(n-2)!\right)^{\sigma}, \quad R_{1, 3}\le C_{4}\left(25A^{k-1}(m-2)!(n-2)!\right)^{\sigma}.
\end{align*}
Combining these results, we have that for all $(m, n)\in E_{k}$
\begin{align*}
     &\frac{4}{\left((m-2)!(n-2)!\right)^{2\sigma}}\int_{0}^{t}\mathcal R_{1}(\tau)d\tau\\
     &\le\frac{1}{16\left((m-2)!(n-2)!\right)^{2\sigma}}\int_{0}^{t}|||F_{m, n}(\tau)|||^{2}d\tau+\frac{4}{\left((m-2)!(n-2)!\right)^{2\sigma}}\left(48(R_{1, 1})^{2}+48(R_{1, 2})^{2}+48(R_{1, 3})^{2}\right)\\
     &\le\sup_{(m, n)\in E_{k}}\frac{1}{16\left((m-2)!(n-2)!\right)^{2\sigma}}\int_{0}^{t}|||F_{m, n}(\tau)|||^{2}d\tau+\left(A_{3}A^{k-1}\right)^{2\sigma},
\end{align*}
this implies
\begin{align*}
R_{1}\le\sup_{(m, n)\in E_{k}}\frac{1}{16\left((m-2)!(n-2)!\right)^{2\sigma}}\int_{0}^{t}|||F_{m, n}(\tau)|||^{2}d\tau+\left(A_{3}A^{k-1}\right)^{2\sigma},
\end{align*}
with the comstant $A_{3}$ depends on $\gamma, \ b$ and $C_{4}$.

\bigskip

{\bf Proof of \eqref{R2}.} Since $(p+q)!\le2^{p+q}p!q!$ for all $p, q\in\mathbb N$, taking $A\ge 2C_{0}(T+1)^{2(\delta_{1}+\delta_{2})}$, then follows from the Cauchy-Schwarz inequality and the fact $\sigma\ge1$ that for all $0<t\le T$ and $(m, n)\in E_{k}$
\begin{align*}
     \int_{0}^{t}\mathcal R_{2}(\tau)d\tau&\le 64\left(4C_{5}\sqrt{T}\sum_{l=0}^{m-1}\sum_{p=0}^{n}\binom{m}{l}\binom{n}{p}A^{\sigma(k-l-p-1)}\sqrt{(m-l+1)!(n-p+2)!}\|F_{l, p}\|_{L^{\infty}(]0, T]; H^{3}_{x}L^{2}_{v})}\right)^{2}\\
     &\quad+64\left(4C_{5}\sqrt{T}\sum_{p=0}^{n-1}\binom{n}{p}A^{\sigma(n-p-1)}\sqrt{(n-p)!}\|F_{m, p}\|_{L^{\infty}(]0, T]; H^{3}_{x}L^{2}_{v})}\right)^{2}\\
     &\quad+\frac{1}{64}\int_{0}^{t}|||F_{m, n}(\tau)|||^{2}d\tau=64(R_{2, 1})^{2}+64(R_{2, 2})^{2}+\frac{1}{64}\int_{0}^{t}|||F_{m, n}(\tau)|||^{2}d\tau,
\end{align*}
For all $(m, n)\in E_{k}$, by using \eqref{m=0}, we can write $A_{2, 1}$ as follows
\begin{align*}
     R_{2, 1}&=4C_{5}\sqrt{T}\sum_{l=0}^{m-1}\sum_{p=1}^{n}\frac{m!((l-2)!)^{\sigma}\sqrt{(m-l+1)!}}{l!(m-l)!}\frac{n!((p-2)!)^{\sigma}\sqrt{(n-p+2)!}}{p!(n-p)!}\\
     &\qquad\times A^{\sigma(k-l-p-1)}\frac{\|F_{l, p}\|_{L^{\infty}(]0, T]; H^{3}_{x}L^{2}_{v})}}{((l-2)!(p-2)!)^{\sigma}}\\
     &\quad+4C_{5}\sqrt{T}\sum_{l=1}^{m-1}\frac{m!((l-2)!)^{\sigma}\sqrt{(m-l+1)!}}{l!(m-l)!}A^{\sigma(k-l-1)}\sqrt{(n+2)!}\frac{\|F_{l, 0}\|_{L^{\infty}(]0, T]; H^{3}_{x}L^{2}_{v})}}{((l-2)!)^{\sigma}}\\
     &\quad+4\epsilon C_{5}B\sqrt{T}A^{\sigma(k-1)}\sqrt{(m-l+1)!(n+2)!},
\end{align*}
since $\sqrt{(p+2)!}\le16(p-2)!$ for all $p\in\mathbb N$, then follows  from \eqref{sum1} that 
\begin{align*}
     \sum_{l=0}^{m-1}\frac{m!((l-2)!)^{\sigma}\sqrt{(m-l+1)!}}{l!(m-l)!}\le16\left(25(m-2)!\right)^{\sigma}, \quad \sum_{p=1}^{n}\frac{n!((p-2)!)^{\sigma}\sqrt{(n-p+2)!}}{p!(n-p)!}\le16\left(25(n-2)!\right)^{\sigma},
\end{align*}
applying the hypothesis \eqref{hypothesis} and taking $\epsilon\le\frac14$, one has for all $(m, n)\in E_{k}$
\begin{align*}
     R_{2, 1}&\le8\cdot16^{2}C_{5}\sqrt{T}\left(25^{2}A^{k-\frac32}(m-2)!(n-2)!\right)^{\sigma}+16^{2}C_{5}B\sqrt{T}\left(25^{2}A^{k-1}(m-2)!(n-2)!\right)^{\sigma}.
\end{align*}
Similarly, one can get that for all $(m, n)\in E_{k}$
\begin{align*}
     R_{2, 2}&\le4\cdot16^{2}C_{5}\sqrt{T}\left(25A^{k-\frac32}(m-2)!(n-2)!\right)^{\sigma}+16^{2}C_{5}B\sqrt{T}\left(25^{2}A^{k-1}(m-2)!(n-2)!\right)^{\sigma}.
\end{align*}
And therefore, for all $(m, n)\in E_{k}$
\begin{align*}
     &\frac{4}{\left((m-2)!(n-2)!\right)^{2\sigma}}\int_{0}^{t}\mathcal R_{2}(\tau)d\tau\\
     &\le\frac{1}{16\left((m-2)!(n-2)!\right)^{2\sigma}}\int_{0}^{t}|||F_{m, n}(\tau)|||^{2}d\tau+\frac{4}{\left((m-2)!(n-2)!\right)^{2\sigma}}\left(64(R_{2, 1})^{2}+48(R_{2, 2})^{2}\right)\\
     &\le\sup_{(m, n)\in E_{k}}\frac{1}{16\left((m-2)!(n-2)!\right)^{2\sigma}}\int_{0}^{t}|||F_{m, n}(\tau)|||^{2}d\tau+\left(A_{4}A^{k-1}\right)^{2\sigma},
\end{align*}
this implies
\begin{align*}
     R_{2}\le\sup_{(m, n)\in E_{k}}\frac{1}{16\left((m-2)!(n-2)!\right)^{2\sigma}}\int_{0}^{t}|||F_{m, n}(\tau)|||^{2}d\tau+\left(A_{4}A^{k-1}\right)^{2\sigma},
\end{align*}
with the comstant $A_{4}$ depends on $\gamma, \ b, T, C_{0}$ and $C_{5}$.

\bigskip

{\bf Proof of \eqref{R3}.} Taking $A\ge 2C_{0}(T+1)^{2(\delta_{1}+\delta_{2})}$, then follows from the Cauchy-Schwarz inequality and the fact $\sigma\ge1$, one has for all $0<t\le T$ and $(m, n)\in E_{k}$
\begin{align*}
     \int_{0}^{t}\mathcal R_{3}(\tau)d\tau&\le38\left(C_{6}\sum_{l=1}^{m}\binom{m}{l}A^{\sigma(l-1)}\sqrt{(l+1)!}\int_{0}^{t}|||F_{m-l, n}(\tau)|||^{2}d\tau\right)^{2}\\
     &\quad+48\left(C_{6}\sum_{p=1}^{n}\binom{n}{p}A^{\sigma(p-1)}\sqrt{(p+1)!}\int_{0}^{t}|||F_{m, n-p}(\tau)|||^{2}d\tau)^{\frac12}\right)^{2}\\
     &\quad+48\left(C_{6}\sum_{l=1}^{m}\sum_{p=1}^{n}\binom{m}{l}\binom{n}{p}A^{\sigma(l+p-1)}\sqrt{l!(p+1)!}\int_{0}^{t}|||F_{m-l, n-p}(\tau)|||^{2}d\tau\right)^{2}\\
     &\quad+\frac{1}{64}\int_{0}^{t}|||F_{m, n}(\tau)|||^{2}d\tau=48(R_{3, 1})^{2}+48(R_{3, 2})^{2}+48(R_{3, 3})^{2}+\frac{1}{64}\int_{0}^{t}|||F_{m, n}(\tau)|||^{2}d\tau.
\end{align*}
Similar to the discussion in $R_{2, 1}$, we can get that for all $(m, n)\in E_{k}$
\begin{align*}
     R_{3, 1}\le8\cdot16^{2}C_{6}\left(25A^{k-\frac32}(m-2)!(n-2)!\right)^{\sigma}, \quad R_{3, 2}\le8\cdot16^{2}C_{6}\left(25A^{k-\frac32}(m-2)!(n-2)!\right)^{\sigma},
\end{align*}
and
\begin{align*}
     R_{3, 3}\le8\cdot16^{2}C_{6}\left(25^{2}A^{k-\frac32}(m-2)!(n-2)!\right)^{\sigma}.
\end{align*}
And therefore, for all $(m, n)\in E_{k}$
\begin{align*}
     &\frac{4}{\left((m-2)!(n-2)!\right)^{2\sigma}}\int_{0}^{t}\mathcal R_{3}(\tau)d\tau\le\sup_{(m, n)\in E_{k}}\frac{1}{16\left((m-2)!(n-2)!\right)^{2\sigma}}\int_{0}^{t}|||F_{m, n}(\tau)|||^{2}d\tau+\left(A_{5}A^{k-1}\right)^{2\sigma},
\end{align*}
this implies
\begin{align*}
     R_{3}\le\sup_{(m, n)\in E_{k}}\frac{1}{16\left((m-2)!(n-2)!\right)^{2\sigma}}\int_{0}^{t}|||F_{m, n}(\tau)|||^{2}d\tau+\left(A_{5}A^{k-1}\right)^{2\sigma},
\end{align*}
with the constant $A_{5}$ depends on $\gamma, \ b$ and $C_{6}$.

\bigskip
\noindent {\bf Acknowledgements.} This work was supported by the NSFC (No.12031006) and the Fundamental
Research Funds for the Central Universities of China.


\begin{thebibliography}{99}


\bibitem{C-L-X} H. Cao, W.-X. Li, and C.-J. Xu, Analytic smoothing effect of the spatially inhomogeneous Landau equations for hard potentials, {\em J. Math. Pures Appl.}, 176 (2023), 138–182.

\bibitem{C-X-X} X.-D. Cao, C.-J. Xu and Y. Xu, Regularizing effect of the spatially homogeneous Landau equation with soft potential, arXiv:2502.12543v1.

\bibitem{CLX1} H. Chen, W.-X. Li, and C.-J. Xu, Analytic smoothness effect of solutions for spatially homogeneous Landau equation, {\em J. Differ. Equ.}, 248(1) (2010), 77-94.

\bibitem{CLX2} H. Chen, W.-X. Li, and C.-J. Xu, Propagation of Gevrey regularity for solutions of Landau equation,  {\em Kinet. Relat. Models}, 1(3) (2008), 355-368.

\bibitem{CLX3} H. Chen, W.-X. Li, and C.-J. Xu, Gevrey regularity for solution of the spatially homogeneous Landau equation, {\em Acta Math. Sci. Ser. B Engl. Ed.}, 29(3) (2009), 673-686.

\bibitem{CLX}  J.-L. Chen, W.-X. Li and C.-J. Xu, Sharp regularization effect for the non-cutoff Boltzmann equation with hard potentials, {\em Ann. Inst. H. Poincaré C Anal. Non Linéaire}, (2024).

\bibitem{CLL} Y. Chen, L. Desvillettes and L. He, Smoothing effects for classical solutions of the full Landau equation, {\em Arch. Ration. Mech. Anal.}, 193 (2009), 21-55.

\bibitem{DV} L. Desvillettes and C. Villani, On the spatially homogeneous Landau equation for hard potentials. I. Existence, uniqueness and smoothness, {\em Commun. Partial Differ. Equ.}, 25(1-2) (2000),  179-259.

\bibitem{DLSR} {R. Duan, S. Liu, S. Sakamoto, and R. M. Strain}, {Global mild solutions of the Landau and non-cutoff Boltzmann equations}, {\em Comm. Pure Appl. Math.}, 74(5) (2021), {932-1020}.

\bibitem{FG} N. Fournier and H. Gu$\rm \acute{e}$rin, Well-posedness of the spatially homogeneous Landau equation for soft potentials, {\em J. Funct. Anal.}, 256 (8) (2009), 2542–2560.

\bibitem{GPR} T. Gramchev, S. Pilipovi${\rm \acute c}$ and L. Rodino, Classes of degenerate elliptic operators in Gelfand- Shilov spaces, {\em New Developments in Pseudo-Differential Operators. Birkh$\ddot a$user Basel}, (2009), 15-31.

\bibitem{G-1} Y. Guo, The Landau Equation in a Periodic Box, {\em Comm. Math. Phys.}, 231 (2002), 391-434.

\bibitem{HJL} L. He, J. Ji and W.-X. Li, On the Boltzmann equation with soft potentials: Existence, uniqueness and smoothing effect of mild solutions, arXiv:2410.13205v1

\bibitem{H-1} C. Henderson and S. Snelson, $C^\infty$ Smoothing for Weak Solutions of the Inhomogeneous Landau Equation, {\em Arch. Ration. Mech. Anal.}, 236(1) (2020), 113-143.

\bibitem{LX} H.-G. Li and C.-J. Xu, Cauchy problem for the spatially homogeneous Landau Equation with Shubin class initial datum and  Gelfand-Shilov smoothing effect, {\rm Siam J. Math. Anal.}, 51(1) (2019), 532-564.

\bibitem{L-1} H.-G. Li and C.-J. Xu, Analytic smoothing effect of non-linear spatially homogeneous Landau equation with hard potentials, {\em Sci. China Math.}, 65 (2022), 2079-2098.

\bibitem{LX-5} H.-G. Li and C.-J. Xu, Analytic Gelfand-Shilov smoothing effect of the spatially homogeneous Landau equation with hard potentials, {\em Discrete and Continuous Dynamical Systems - B.}, 29(4), (2024), 1815-1840.

\bibitem{LX2} H.-G. Li and C.-J. Xu, Analytic smoothing effect of the linear Landau equation with soft potential, {\em Acta Mathematica Scientia. Series B.English Edition}, (2023), 2597-2614.

\bibitem{LX3} H.-G. Li and C.-J. Xu, Gelfand–Shilov smoothing effect of the spatially homogeneous Landau equation with moderately soft potential, {\em Math. Meth. Appl. Sci.}, (2023), 1-28. DOI 10.1002/mma.9325

\bibitem{M-1} Y. Morimoto, K. Pravda-Starov, and C.-J. Xu, A remark on the ultra-analytic smoothing properties of the spatially homogeneous Landau equation, {\em Kinet. Relat. Models}, 6(4) (2013), 715-727.

\bibitem{MX} Y. Morimoto and C.-J. Xu, Ultra-analytic effect of Cauchy problem for a class of kinetic equations, {\em J. Differ. Equ.}, 247(2) (2009), 596-617.

\bibitem{M-2} Y. Morimoto and C.-J. Xu, Analytic smoothing effect of the nonlinear Landau equation of Maxwellian molecules, {\em Kinet. Relat. Models}, 13(5) (2020), 951-978.

\bibitem{V1} C. Villani, On the spatially homogeneous Landau equation for Maxwellian molecules, {\em Mathematical Methods and Methods in Applied Sciences}, 8(6) (1998), 957-983.

\bibitem{V2} C. Villani, On a new class of weak solutions to the spatially homogeneous Boltzmann and Landau equations, {\em Arch. Ration. Mech. Anal.}, 143 (3) (1998), 273-307.

\bibitem{W} K.-C. Wu, Global in time estimates for the spatially homogeneous Landau equation with soft potentials, {\em J. Funct. Anal.}, 266 (2014), 3134-3155.

\bibitem{XX}  C.-J. Xu and Y. Xu, A remark about time-analyticity of the linear Landau equation with soft potential, {\em Anal. Theory Appl.}, 40(1) (2024), 22-37.

\bibitem{X} C.-J. Xu and Y. Xu, The analytic Gelfand-Shilov smoothing effect of the Landau equation with hard potential, {\em J. Differ. Equ.}, 414 (2024), 645-681.

\end{thebibliography}
\end{document}